\documentclass[11pt,twoside]{article}
\usepackage[latin1]{inputenc}
\usepackage{mathrsfs}
\usepackage{amssymb}
\usepackage{amsmath}

\usepackage{color}
\usepackage[all]{xy}
\usepackage{amsthm}

\usepackage{url}
\usepackage[numbers,sort&compress]{natbib}
\usepackage{mathtools}

\makeatletter
\DeclareRobustCommand\widecheck[1]{{\mathpalette\@widecheck{#1}}}
\def\@widecheck#1#2{%
    \setbox\z@\hbox{\m@th$#1#2$}%
    \setbox\tw@\hbox{\m@th$#1%
       \widehat{%
          \vrule\@width\z@\@height\ht\z@
          \vrule\@height\z@\@width\wd\z@}$}%
    \dp\tw@-\ht\z@
    \@tempdima\ht\z@ \advance\@tempdima2\ht\tw@ \divide\@tempdima\thr@@
    \setbox\tw@\hbox{%
       \raise\@tempdima\hbox{\scalebox{1}[-1]{\lower\@tempdima\box
\tw@}}}%
    {\ooalign{\box\tw@ \cr \box\z@}}}
\makeatother

\numberwithin{equation}{section}

\newtheorem{theorem}{Theorem}[section]
\newtheorem{proposition}[theorem]{Proposition}
\newtheorem{lemma}[theorem]{Lemma}

\newtheorem{corollary}[theorem]{Corollary}

\newtheorem{theorem*}{Theorem}

\theoremstyle{definition}
\newtheorem{definition}[theorem]{Definition}
\newtheorem{example}[theorem]{Example}

\theoremstyle{remark}
\newtheorem{remark}[theorem]{Remark}
\usepackage{paralist}

\usepackage{indentfirst}
\newcommand{\Hom}{\operatorname{Hom}}
\newcommand{\Ext}{\operatorname{Ext}}

\newcommand{\id}{\operatorname{id}}

\newcommand{\res}{\operatorname{res.dim}}
\newcommand{\cores}{\operatorname{cores.dim}}
\newcommand{\ra}{\rightarrow}

\def\pd{\mathop{\rm pd}\nolimits}

\def\id{\mathop{\rm id}\nolimits}

\def\dim{\mathop{\rm res.dim}\nolimits}

\newcommand{\s}{\mathfrak{s}}
\newcommand{\E}{\mathbb{E}}

\newcommand{\Id}{\operatorname{Id}}

\newcommand{\RNum}[1]{\uppercase\expandafter{\romannumeral #1\relax}}

\title{\Large \bf One-sided Frobenius pairs  in extriangulated categories
}
\vspace{0.2cm}

\author{Lingling Tan$^a$, Yuqiong Gao$^a$, Qinghua Chen$^b$\thanks{E-mail address: tanll@qfnu.edu.cn, gaoyqmath@163.com, chenqh@fzu.edu.cn}
\\
{\footnotesize a. School of Mathematical Sciences, Qufu Normal University, Qufu 273165, China,}\\
\footnotesize b. School of Mathematics and Statistics, Fu Zhou University, Fu Zhou, Fujian
350108,  China
}
\date{ }
\begin{document}

\baselineskip=16pt
\maketitle

\begin{abstract}
Let $\mathscr{C}$ be an extriangulated category with a proper class $\xi$ of $\mathbb{E}$-triangles.
We introduce the notions of left Frobenius pairs, left ($n$-)cotorsion pairs and left
(weak) Auslander-Buchweitz contexts with respect to $\xi$ in $\mathscr{C}$. We show how to construct left cotorsion pais from left $n$-cotorsion pairs, and establish a one-to-one correspondence between left Frobenius pairs and left (weak) Auslander-Buchweitz contexts. We also study the relation between a certain class of cotorsion pairs and that of $n$-cotorsion pairs.   These work generalize Ma-Zhao-Huang's results in triangulated categories and partially generalize Becerril-Mendoza-P\'{e}rez-Santiago's results in abelian categories.
\end{abstract}

\section{Introduction}
Auslander-Buchweitz approximation theory was firstly developed by Auslander and Buchweitz  \cite{AB89T}.
Later on, Hashimoto \cite{H00A} investigated the Auslander-Buchweitz context for abelian categories.
Cotorsion pairs  were introduced by  Salce \cite{S} in the category of
abelian groups, and developed in the study of
the algebraic and geometric structures of abelian categories (see \cite{ET01H, EJ01R,GT06A}). This notion provides a good setting for
investigating relative homological dimensions (see \cite{AM}). As a higher version of cotorsion pairs,
Huerta et al. \cite{HMP} introduced the notion of $n$-cotorsion pairs in abelian categories.
They described several properties of $n$-cotorsion pairs and
established a relation with (complete) cotorsion pairs. As a remarkable study,  another
higher version of cotorsion pair in abelian categories was previously
introduced by Crivei and Torrecillas  \cite{CT}. On the other hand,
Becerril et al. \cite{BMP19F} introduced Frobenius pairs in abelian categories.
They presented one-to-one correspondences between left Frobenius pairs, Auslander-Buchweitz contexts and
cotorsion pairs in abelian categories. Recently, Ma et al. \cite{MZ} generalized partially this correspondence to triangulated categories with a proper class of triangles.

In \cite{NP}, Nakaoka and Palu introduced the notion of extriangulated categories
 as a simultaneous generalization of exact categories and  triangulated
categories. Up to now,  many results on exact categories and triangulated
categories have gotten realization in the setting of extriangulated categories, e.g. \cite{HZZ,LN,NOS,NP,ZZ,ZZh}, etc. In particular, Ma et al. \cite{MDZ} investigated Auslander-Buchweitz approximation theory in extriangulated categories. In this paper, we will pay close attention to the notion of a proper class $\xi$ of $\mathbb{E}$-triangles in an extriangulated category $(\mathscr{C},\mathbb{E},\mathfrak{s})$, which is introduced by Hu et al. \cite{HZZ} for developing the Gorenstein homological algebra over extriangulated categories.  Throughout this paper, we always assume that $\mathscr{C}=(\mathscr{C},\mathbb{E},\mathfrak{s})$ is an extriangulated category and $\xi$ is a proper class of $\mathbb{E}$-triangles. We also assume that $\mathscr{C}$ has
 enough $\xi$-projective and $\xi$-injective objects.
We are devoted to introducing the notions of  left Frobenius pairs, left ($n$-)cotorsion pairs and left
(weak) Auslander-Buchweitz contexts with respect to  a proper class $\xi$ of $\mathbb{E}$-triangles in an extriangulated category $\mathscr{C}$, and  developing relative homological theory along with the
Auslander-Buchweitz approximation theory. In particular, we will discuss the internal relations among these notions.
Moreover, some applications are given in the context of Gorenstein homological algebra
in extriangulated categories. When $\mathscr{C}$ is a triangulated category with a proper class of triangles, it recovers the results of Ma et al.  \cite{MZ}, and when $\mathscr{C}$ is an abelian category, it recovers  some of results
 of  Becerril et al. \cite{BMP19F} and Huerta et al. \cite{HMP}. Recently, Adachi and Tsukamoto \cite{AT} investigated the relation between silting subcategories and
 left Frobenius pairs over extriangulated categories.
This paper is organized as follows.

In Section $2$, we give some terminology and some preliminary results.

In Section $3$, we recall the notion of left (resp. right) $n$-cotorsion pairs
with respect to $\xi$ in $\mathscr{C}$, and then by virtue of an equivalent characterization of $n$-cotorsion pairs in
\cite{ZPY}, we establish a relation between $n$-cotorsion pairs and cotorsion pairs
(Proposition \ref{prop-6}).

In Section $4$, we introduce the notions of left Frobenius pairs and left (weak) Auslander-Buchweitz
contexts with respect to $\xi$ in  $\mathscr{C}$. For a subcategory $\mathcal{X}$ of $\mathscr{C}$,
$\mathcal{X}^{\wedge}$ denotes the subcategory of $\mathscr{C}$ consisting of objects with finite
$\mathcal{X}$-resolution dimension.
Let $(\mathcal{X},\omega)$ be a left Frobenius pair  in $\mathscr{C}$. We show that
$\mathcal{X}^{\wedge}$ is closed under $\xi$-extensions, cocones of $\xi$-deflations, cones
of $\xi$-inflations and direct summands (Theorem \ref{thm-4.9}).
Then we show how to obtain (left) cotorsion pairs from left Frobenius pairs (Theorem \ref{thm-ftoc}).
Finally, we introduce the notion of left (weak) Auslander-Buchweitz contexts.

 Let $\mathfrak{A}$ (resp. $\mathfrak{A}'$) be the class of all left Frobenius pairs $(\mathcal{X},\omega)$ with respect to $\xi$ (resp. with $\mathcal{X}^\wedge=\mathscr{C}$) in $\mathscr{C}$, and
 $\mathfrak{B} $  (resp. $\mathfrak{B}'$) be the class of all left weak (resp. left) Auslander-Buchweitz context $(\mathcal{A},\mathcal{B})$ with respect to $\xi$ in  $\mathscr{C}$.
Then we establish  a one-to-one correspondence  (Theorems \ref{main} and \ref{cor}) between $\mathfrak{A}$ and $\mathfrak{B}$ (resp. $\mathfrak{A}'$ and $\mathfrak{B}'$)
given by
\begin{align*}
\Phi:&~\mathfrak{A} ~ ({\rm resp. }~ \mathfrak{A}')~\longrightarrow~\mathfrak{B}  ~ ({\rm resp. } ~\mathfrak{B}')\ via\ (\mathcal{X},\omega)\mapsto (\mathcal{X},\ {\omega}^{\wedge})\\
\Psi:&~\mathfrak{B} ~ ({\rm resp. } ~\mathfrak{B}')\longrightarrow\mathfrak{A} ~ ({\rm resp. }~ \mathfrak{A}')~\ via\ (\mathcal{A},\mathcal{B})\mapsto (\mathcal{A},\ \mathcal{A}\cap\mathcal{B}).
\end{align*}
Let $\mathfrak{C}'$ be the class of all cotorsion pairs $(\mathcal{U},\mathcal{V})$ with respect to $\xi$ in
$\mathscr{C}$  with $\mathcal{U}$ resolving and $\mathcal{U}^{\wedge}=\mathscr{C}$, and $\mathfrak{D}'$ be the class of all $n$-cotorsion pairs $ (\mathcal{U},\mathcal{V})$ with respect to $\xi$ in $\mathscr{C}$  with $\mathcal{U}$  resolving and $\mathcal{U}^{\wedge}=\mathscr{C}$.
We also show that $\mathfrak{B}'=\mathfrak{C}'=\mathfrak{D}'$ (Theorem \ref{cor}).

\section{Preliminaries}\label{P}

We first recall some notions and some needed properties of
extriangulated categories from \cite{NP}.

Let $\mathscr{C}$ be an additive category and $\E:\mathscr{C}^{\operatorname{op}}\times\mathscr{C}\rightarrow \mathfrak{A}b$
 a biadditive functor, where $\mathfrak{A}b$ is the category of abelian groups.
Let $A,C\in\mathscr{C}$. An element $\delta\in \mathbb{E}(C,A)$ is called an \emph{$\mathbb{E}$-extension}.
Two sequences of morphisms
$$
\xymatrix@C=0.5cm{
  A \ar[r]^x & B \ar[r]^y & C }
  \mbox{ and }
  \xymatrix@C=0.5cm{
  A \ar[r]^{x'} & B' \ar[r]^{y'} & C}
$$
are said to be \emph{equivalent} if there exists an isomorphism $b\in\Hom_{\mathscr{C}}(B,B')$ such that $x'=bx$ and $y=y'b$.
We denote by $[\xymatrix@C=0.5cm{
A \ar[r]^x & B \ar[r]^y & C }]$ the equivalence class of $\xymatrix@C=0.5cm{
A \ar[r]^x & B \ar[r]^y & C}$. In particular, we write $0:=[\xymatrix@C=0.7cm{
A \ar[r]^{\!\!\!\!\!\!\!{\Id_A\choose 0}} & A\oplus C \ar[r]^{~~~(0 \ \Id_C)} & C }]$.

For an $\E$-extension $\delta\in \mathbb{E}(C,A)$, we briefly write
$$a_*\delta:=\E(C,a)(\delta)\ {\rm and}\ c^*\delta:=\E(c,A)(\delta).$$
For two $\E$-extensions $\delta\in \mathbb{E}(C,A)$ and $\delta'\in \mathbb{E}(C',A')$,
a \emph{morphism} from $\delta$ to $\delta'$ is a pair $(a,c)$ of morphisms with
$a\in \Hom_{\mathscr{C}}(A,A')$ and $c\in \Hom_{\mathscr{C}}(C,C')$ such that $a_*\delta=c^*\delta'$.

\begin{definition} (\cite[Definition 2.9]{NP})
Let $\s$ be a correspondence which associates an equivalence class $\s(\delta)=[\xymatrix@C=0.5cm{
A \ar[r]^x & B \ar[r]^y & C }]$ to each $\E$-extension $\delta\in\E(C,A)$. Such $\s$ is called a
\emph{realization} of $\E$ provided that it satisfies the following condition.
\begin{enumerate}
\item[(R)]  Let $\delta\in\E(C,A)$ and $\delta'\in\E(C',A')$ be any pair of $\E$-extensions with
$$\s(\delta)=[\xymatrix@C=0.5cm{
A \ar[r]^x & B \ar[r]^y & C }]
\mbox{ and }
\s(\delta')=[ \xymatrix@C=0.5cm{
A' \ar[r]^{x'} & B' \ar[r]^{y'} & C' }].$$
Then for any morphism $(a,c):\delta\rightarrow\delta'$, there exists $b\in\Hom_{\mathscr{C}}(B,B')$ such that the following diagram
$$\xymatrix@=0.5cm{
A \ar[r]^x\ar[d]^a & B \ar[r]^y\ar[d]^b & C\ar[d]^c\\
A' \ar[r]^{x'} & B' \ar[r]^{y'} & C' }$$
commutes.
\end{enumerate}
Let $\s$ be a realization of $\E$. If $\s(\delta)=[\xymatrix@C=0.5cm{
A \ar[r]^x & B \ar[r]^y & C }]$ for some $\E$-extension $\delta\in\E(C,A)$, then we say that the sequence
$\xymatrix@C=0.5cm{A \ar[r]^x & B \ar[r]^y & C }$ \emph{realizes} $\delta$; and in the condition (R),
we say that the triple $(a,b,c)$ \emph{realizes} the morphism $(a,c)$.
\end{definition}

For any two equivalence classes $[\xymatrix@C=0.5cm{
  A \ar[r]^x & B \ar[r]^y & C }]$ and $[\xymatrix@C=0.5cm{
  A' \ar[r]^{x'} & B' \ar[r]^{y'} & C' }]$, we define
$$[\xymatrix@C=0.5cm{
  A \ar[r]^x & B \ar[r]^y & C }]\oplus[\xymatrix@C=0.5cm{
  A' \ar[r]^{x'} & B' \ar[r]^{y'} & C' }]:=[\xymatrix@C=0.5cm{
  A\oplus A' \ar[r]^{x\oplus x'} & B\oplus B' \ar[r]^{y\oplus y'} & C\oplus C' }].$$

\begin{definition} (\cite[Definition 2.10]{NP})
A realization $\s$ of $\E$ is called \emph{additive} if it satisfies the following conditions.
\begin{enumerate}
\item[(1)] For any $A,C\in\mathscr{C}$, the split $\E$-extension $0\in\E(C,A)$ satisfies $\s(0)=0$.
\item[(2)] For any pair of $\E$-extensions $\delta\in\E(C,A)$ and $\delta'\in\E(C',A')$,
we have $\s(\delta\oplus\delta')=\s(\delta)\oplus\s(\delta')$.
\end{enumerate}
\end{definition}

\begin{definition}{ {(\cite[Definitions 2.15 and 2.19]{NP})}} {

Let $\mathscr{C}$ be an additive category,~$\mathfrak{s}$ be an additive realization.
\begin{enumerate}
\item A sequence $\xymatrix@C=0.5cm{A\ar[r]^x&B\ar[r]^{y}&C}$ is called a {\it conflation} if it realizes some $\mathbb{E}$-extension $\delta\in\mathbb{E}(C, A)$.
In this case, $x$ is called an {\it inflation} and $y$ is called a {\it deflation}.

\item  If a conflation $\xymatrix@C=0.5cm{A\ar[r]^x&B\ar[r]^{y}&C}$ realizes $\delta\in\mathbb{E}(C, A)$, we call the pair
$\xymatrix@C=0.5cm{(A\ar[r]^x&B\ar[r]^{y}&C, \delta)}$ an {\it $\mathbb{E}$-triangle}, and write it in the following way.
\begin{center} $\xymatrix@C=0.5cm{A\ar[r]^x&B\ar[r]^{y}&C\ar@{-->}[r]^{\delta}&}$\end{center}
We usually do not write this ``$\delta$" if it is not used in the argument.

\item Let $\xymatrix@C=0.5cm{A\ar[r]^x&B\ar[r]^{y}&C\ar@{-->}[r]^{\delta}&}$ and $\xymatrix@C=0.5cm{A'\ar[r]^{x'}&B'\ar[r]^{y'}&C'\ar@{-->}[r]^{\delta'}&}$
be any pair of $\mathbb{E}$-triangles. If a triplet $(a, b, c)$ realizes $(a, c): \delta\rightarrow \delta'$, then we write it as
 $$\xymatrix@C=0.5cm@R=0.4cm{A\ar[r]^{x}\ar[d]_{a}&B\ar[r]^{y}\ar[d]_{b}&C\ar[d]_{c}\ar@{-->}[r]^{\delta}&\\
 A'\ar[r]^{x'}&B'\ar[r]^{y'}&C'\ar@{-->}[r]^{\delta'}&}$$
 and call $(a, b, c)$ a {\it morphism} of $\mathbb{E}$-triangles.

If $a, b, c$ above are isomorphisms, then $\xymatrix@C=0.4cm{A\ar[r]^x&B\ar[r]^{y}&C\ar@{-->}[r]^{\delta}&}$ and $\xymatrix@C=0.4cm{A'\ar[r]^{x'}&B'\ar[r]^{y'}&C'\ar@{-->}[r]^{\delta'}&}$ are said to be {\it isomorphic}.
\end{enumerate}}

\end{definition}
\begin{definition} (\cite[Definition 2.12]{NP})
The triple $(\mathscr{C},\E,\s)$ is called an \emph{externally triangulated}
(or \emph{extriangulated} for short) category if it satisfies the following conditions.
 \begin{enumerate}
\item[(ET1)] $\E:\mathscr{C}^{\operatorname{op}}\times\mathscr{C}\rightarrow \mathfrak{A}b$
is a biadditive functor.
\item[(ET2)] $\s$ is an additive realization of $\E$.
\item[(ET3)] Let $\xymatrix@C=0.5cm{
A\ar[r]^{x} &B \ar[r]^{y} &C \ar@{-->}[r]^{\delta}&}$ and$\xymatrix@C=0.5cm{
A'\ar[r]^{x} &B' \ar[r]^{y} &C'\ar@{-->}[r]^{\delta'}&}$ be any pair of $\E$-triangles.
For any  diagram with $bx=x'a$
$$
\xymatrix@=0.5cm{
   A \ar[r]^x\ar[d]^a & B \ar[r]^y\ar[d]^b & C\ar@{.>}[d]^c \ar@{-->}[r]^{\delta}&\\
    A' \ar[r]^{x'} & B' \ar[r]^{y'} & C' \ar@{-->}[r]^{\delta'}& }
$$
in $\mathscr{C}$, there exists a morphism $(a,c):\delta\ra \delta'$ which is realized by the triple $(a,b,c)$, that is, $(a,b,c)$ is a morphism of $\E$-triangles.
\item[${\rm (ET3)^{{op}}}$]   Dual of (ET3).
\item[(ET4)]   Let $\xymatrix@C=0.5cm{
A\ar[r]^{x} &B \ar[r]^{y} &C \ar@{-->}[r]^{\delta}&}$ and$\xymatrix@C=0.5cm{
B\ar[r]^{u} &D \ar[r]^{v} &F\ar@{-->}[r]^{\rho}&}$ be any pair of $\E$-triangles.
Then there exist an object $E\in\mathscr{C}$, an $\E$-triangle $\xymatrix@C=0.5cm{
  A \ar[r]^z & D \ar[r]^w & E \ar@{-->}[r]^{\xi} & }$, and a commutative diagram
$$
\xymatrix@=0.5cm{
   A \ar[r]^x\ar@{=}[d] & B \ar[r]^y\ar[d]^u & C\ar@{.>}[d]^s\ar@{-->}[r]^{\delta}&\\
    A \ar@{.>}[r]^{z} & D \ar@{.>}[r]^{w}\ar[d]^v &E\ar@{.>}[d]^t\ar@{-->}[r]^\xi&\\
    &F\ar@{-->}[d]^{\rho}\ar@{=}[r] &F\ar@{-->}[d]\\&&& }
$$
in $\mathscr{C}$, which satisfy the following compatibilities.
 \begin{enumerate}
\item[(i)] $\xymatrix@C=0.5cm{
  C \ar[r]^s & E \ar[r]^t & F \ar@{-->}[r]^{y_*\rho}&}$ is an $\E$-triangle.
\item[(ii)] $s^*\xi=\delta$.
\item[(iii)] $x_*\xi=t^*\rho$.
\end{enumerate}
\item[${\rm (ET4)^{op}}$]    Dual of (ET4).
\end{enumerate}
\end{definition}

\begin{remark}
Note that both exact categories and triangulated categories are extriangulated categories $($see \cite[Proposition 3.22]{NP}$)$ and extension closed subcategories of extriangulated categories are
again extriangulated $($see \cite[Remark 2.18]{NP}$)$. Moreover, there exist extriangulated categories which
are neither exact categories nor triangulated categories $($see \cite[Proposition 3.30]{NP} and \cite[Remark 3.3]{HZZ}$)$.
\end{remark}

The following condition is analogous to the weak idempotent completeness in exact categories (see \cite[Condition 5.8]{NP}).

 {\bf Condition (WIC)}  Consider the following conditions.

\begin{enumerate}
\item[(a)]  Let $f\in\mathscr{C}(A, B), g\in\mathscr{C}(B, C)$ be any composable pair of morphisms. If $gf$ is an inflation, then so is $f$.

\item[(b)] Let $f\in\mathscr{C}(A, B), g\in\mathscr{C}(B, C)$ be any composable pair of morphisms. If $gf$ is a deflation, then so is $g$.

\end{enumerate}

\begin{example}\label{Ex:4.12}

\emph{(1)} If $\mathscr{C}$ is an exact category, then Condition {(WIC)} is equivalent to that $\mathscr{C}$ is
weakly idempotent complete {(see \cite[Proposition 7.6]{B"u})}.

In detail, $\mathscr{C}$ is weakly idempotent complete if and only if given two morphisms $g:B\rightarrow C$ and $f:A\rightarrow B$, if $gf:A\rightarrow C$ is a deflation, then $g$ is a deflation.

\emph{(2)} If $\mathscr{C}$ is a triangulated category, then Condition {(WIC)} is automatically satisfied.
\end{example}

\begin{lemma}\label{lem1} {\rm (\cite[Proposition 3.15]{NP})} Assume that $(\mathscr{C}, \mathbb{E},\mathfrak{s})$ is an extriangulated category.
\begin{enumerate}
\item[(1)]
Let $C$ be an object in $\mathscr{C}$, and let $\xymatrix@C=0.5cm{A_1\ar[r]^{x_1}&B_1\ar[r]^{y_1}&C\ar@{-->}[r]^{\delta_1}&}$ and $\xymatrix@C=0.5cm{A_2\ar[r]^{x_2}&B_2\ar[r]^{y_2}&C\ar@{-->}[r]^{\delta_2}&}$ be any pair of $\mathbb{E}$-triangles. Then there is a commutative diagram
in $\mathscr{C}$
$$\xymatrix@=0.5cm{
    & A_2\ar@{.>}[d]_{m_2} \ar@{=}[r] & A_2 \ar[d]^{x_2} \\
  A_1 \ar@{=}[d] \ar@{.>}[r]^{m_1} & M \ar@{.>}[d]_{e_2} \ar@{.>}[r]^{e_1} & B_2\ar[d]^{y_2}\ar@{-->}[r]& \\
  A_1 \ar[r]^{x_1} & B_1\ar[r]^{y_1}\ar@{-->}[d] & C\ar@{-->}[r]^{\delta_{1}}\ar@{-->}[d]^{\delta_{2}}&\\&&&  }
  $$
  which satisfies $\mathfrak{s}(y^*_2\delta_1)=\xymatrix@C=0.5cm{[A_1\ar[r]^{m_1}&M\ar[r]^{e_1}&B_2]}$ and
  $\mathfrak{s}(y^*_1\delta_2)=\xymatrix@C=0.5cm{[A_2\ar[r]^{m_2}&M\ar[r]^{e_2}&B_1].}$ 
\item[(2)] Dual to (1).
\end{enumerate}
\end{lemma}

The following definitions are quoted verbatim from \cite[Section 3]{HZZ}.
\begin{definition}
A class of $\mathbb{E}$-triangles $\xi$ is {\it closed under base change} if for any $\mathbb{E}$-triangle $$\xymatrix@C=0.5cm{A\ar[r]^x&B\ar[r]^y&C\ar@{-->}[r]^{\delta}&\in\xi}$$ and any morphism $c\colon C' \to C$, then any $\mathbb{E}$-triangle  $\xymatrix@C=0.5cm{A\ar[r]^{x'}&B'\ar[r]^{y'}&C'\ar@{-->}[r]^{c^*\delta}&}$ belongs to $\xi$.

Dually, a class of  $\mathbb{E}$-triangles $\xi$ is {\it closed under cobase change} if for any $\mathbb{E}$-triangle $$\xymatrix@C=0.5cm{A\ar[r]^x&B\ar[r]^y&C\ar@{-->}[r]^{\delta}&\in\xi}$$ and any morphism $a\colon A \to A'$, then any $\mathbb{E}$-triangle  $\xymatrix@C=0.5cm{A'\ar[r]^{x'}&B'\ar[r]^{y'}&C\ar@{-->}[r]^{a_*\delta}&}$ belongs to $\xi$.

A class of $\mathbb{E}$-triangles $\xi$ is called {\it saturated} if in the situation of Lemma \ref{lem1}(1), whenever
$\xymatrix@C=0.5cm{A_2\ar[r]^{x_2}&B_2\ar[r]^{y_2}&C\ar@{-->}[r]^{\delta_2 }&}$
 and $\xymatrix@C=0.5cm{A_1\ar[r]^{m_1}&M\ar[r]^{e_1}&B_2\ar@{-->}[r]^{y_2^{\ast}\delta_1}&}$
 belong to $\xi$, then the  $\mathbb{E}$-triangle $$\xymatrix@C=0.5cm{A_1\ar[r]^{x_1}&B_1\ar[r]^{y_1}&C\ar@{-->}[r]^{\delta_1 }&}$$  belongs to $\xi$.

An $\mathbb{E}$-triangle $\xymatrix@C=0.5cm{A\ar[r]^x&B\ar[r]^y&C\ar@{-->}[r]^{\delta}&}$ is called {\it split} if $\delta=0$. It is easy to see that it is split if and only if $x$ is section or $y$ is retraction, that is, there exists $r\in\mathscr{C}(B,A), s\in\mathscr{C}(C,B)$, which satisfies ${\footnotesize \begin{bmatrix} r\\y\end{bmatrix}}: \xymatrix@C=0.5cm{B\ar[r]^-{\cong}&A\bigoplus C}$.
\end{definition}

The full subcategory  consisting of the split $\mathbb{E}$-triangles will be denoted by $\Delta_0$.

  \begin{definition} {(cf. \cite[Definition 3.1]{HZZ})}\label{def:proper class} {\rm  Let $\xi$ be a class of $\mathbb{E}$-triangles which is closed under isomorphisms.  If the following conditions hold:
  \begin{enumerate}
\item[(a)]  $\xi$ is closed under finite coproducts and $\Delta_0\subseteq \xi$,

\item[(b)] $\xi$ is closed under base change and cobase change,

\item[(c)] $\xi$ is saturated,
  \end{enumerate}
  then we call $\xi$  a {\it proper class} of $\mathbb{E}$-triangles.}
  \end{definition}

 \begin{definition} {(\cite[Definition 4.1]{HZZ})}
 {\rm An object $P\in\mathscr{C}$  is called {\it $\xi$-projective}  if for any $\mathbb{E}$-triangle $\xymatrix@=0.5cm{
 A\ar[r]^x& B\ar[r]^y& C \ar@{-->}[r]^{\delta}&
 }$ in $\xi$, the induced sequence of abelian groups $$\xymatrix@C=0.5cm{0\ar[r]& \Hom_\mathscr{C}(P,A)\ar[r]& \Hom_\mathscr{C}(P,B)\ar[r]&\Hom_\mathscr{C}(P,C)\ar[r]& 0}$$ is exact. Dually, we have the definition of {\it $\xi$-injective} objects.}
\end{definition}

We denote by $\mathcal{P}(\xi)$ (resp. $\mathcal{I}(\xi)$) the full subcategory of $\mathscr{C}$ consisting of  $\xi$-projective (resp., $\xi$-injective) objects. It follows from the definition that  $\mathcal{P}(\xi)$ and $\mathcal{I}(\xi)$ are full, additive, closed under isomorphisms and direct summands.

 An extriangulated  category $(\mathscr{C}, \mathbb{E}, \mathfrak{s})$ is said to  have {\it  enough
$\xi$-projectives} \ (resp., {\it  enough $\xi$-injectives}) provided that for each object $A$ there exists an $\mathbb{E}$-triangle $\xymatrix@C=0.5cm{K\ar[r]& P\ar[r]&A\ar@{-->}[r]& }$ (resp., $\xymatrix@C=0.5cm{A\ar[r]& I\ar[r]& K\ar@{-->}[r]&}$) in $\xi$ with $P\in\mathcal{P}(\xi)$ (resp., $I\in\mathcal{I}(\xi)$).

The {\it $\xi$-projective dimension} $\xi$-${\rm pd} A$ of $A\in\mathscr{C}$ is defined inductively.
 If $A\in\mathcal{P}(\xi)$, then define $\xi$-${\rm pd} A=0$. For a positive integer $n$, one writes $\xi$-${\rm pd} A= n$ provided:
\begin{enumerate}
\item[(a)] there is an $\mathbb{E}$-triangle $\xymatrix@C=0.5cm{K\ar[r]& P\ar[r]&A\ar@{-->}[r]& }$ with $P\in\mathcal{P}(\xi)$ and $\xi$-${\rm pd} K= n-1$,
\item[(b)] there does not exist an $\mathbb{E}$-triangle $\xymatrix@C=0.5cm{L\ar[r]& P'\ar[r]&A\ar@{-->}[r]& }$ with $P'\in\mathcal{P}(\xi)$ and $\xi$-${\rm pd} L< n-1$.
\end{enumerate}
 We set $\xi$-${\rm pd} A=\infty$, if $\xi$-${\rm pd} A\neq n$ for all $n\geq 0$.

Dually, we can define the {\it $\xi$-injective dimension}  $\xi$-${\rm id} A$ of an object $A\in\mathscr{C}$.

\begin{definition} {(\cite[Definition 4.4]{HZZ})}
 A {\it $\xi$-exact} complex $\mathbf{X}$ is a diagram $$\xymatrix@C=0.5cm{
 \cdots\ar[r]&X_1\ar[r]^-{d_1}&X_0\ar[r]^-{d_0}&X_{-1}\ar[r]&\cdots}$$ in $\mathscr{C}$ such that for each integer $n$, we have $d_n=g_{n-1}f_n$ for some   $\mathbb{E}$-triangle $$\xymatrix@C=0.5cm{
 K_{n+1}\ar[r]^-{g_n}&X_n\ar[r]^{f_n}&K_n\ar@{-->}[r]&}$$ in $\xi$.

In particular, by saying that $$\xymatrix@C=0.5cm{
X_{n}\ar[r]^-{d_n}&X_{n-1}\ar[r]&\cdots\ar[r]&X_{1}
\ar[r]^{d_1}&X_{0}}$$ is $\xi$-exact, it means that there are $\mathbb{E}$-triangles
$$\xymatrix@C=0.5cm{
X_n\ar[r]^-{d_n}&X_{n-1}\ar[r]^{f_{n-1}}&K_{n-1}\ar@{-->}[r]&} \mbox{ and } \xymatrix@C=0.5cm{
K_2\ar[r]^{g_1}&X_{1}\ar[r]^{d_{1}}&X_0\ar@{-->}[r]&}$$  in $\xi$, and  for each integer $1<i<n-1$, we have $d_i=g_{i-1}f_i$ for some   $\mathbb{E}$-triangle $$\xymatrix@C=0.5cm{
K_{i+1}\ar[r]^-{g_{i}}&X_i\ar[r]^{f_i}&K_i\ar@{-->}[r]&}$$ in $\xi$.
\end{definition}

\begin{definition} {(\cite[Definition 3.1]{HZZ1})}\label{df:resolution}  Let $M$ be an object in $\mathscr{C}$. By a  {\it $\xi$-projective resolution} of $M$ we mean a symbol of the form $\mathbf{P}\rightarrow M$ where $\mathbf{P}$ is a  $\xi$-exact complex,  $P_n\in{\mathcal{P}(\xi)}$ for all $n\geq0$, and  $P_{-1}=M$ and $P_n=0$ for all $n<-1$.

The notion of {\it $\xi$-injective coresolution} of $M$ is given dually.
\end{definition}

\begin{definition} {(\cite[Definition 3.2]{HZZ1})}\label{df:derived-functors} { Let $M$ and $N$ be objects in $\mathscr{C}$.

\begin{enumerate}
\item[{\rm (1)}] If we choose a $\xi$-projective resolution $\xymatrix@C=0.5cm{\mathbf{P}\ar[r]& M}$ of  $M$, by applying the functor $\Hom_\mathscr{C}(-,N)$ to $\mathbf{P}$ we have a complex of abelian groups ${\Hom_\mathscr{C}}(\mathbf{P},N)$.  For any integer $n\geq 0$, the \emph{$\xi$-cohomology groups} $\xi{\rm xt}_{\mathcal{P}(\xi)}^n(M,N)$ are defined as
$$\xi{\rm xt}_{\mathcal{P}(\xi)}^n(M,N)=H^n({\Hom_\mathscr{C}}(\mathbf{P},N)).$$

\item[{\rm (2)}] If we choose a
$\xi$-injective coresolution $\xymatrix@C=0.5cm{N\ar[r]&\mathbf{I}}$ of  $N$, by applying the functor $\Hom_\mathscr{C}(M,-)$ to $\mathbf{I}$ we have a complex of abelian groups ${\Hom_\mathscr{C}}(M,\mathbf{I})$.  For any integer $n\geq 0$, the \emph{$\xi$-cohomology groups} $\xi{\rm xt}_{\mathcal{I}(\xi)}^n(M,N)$ are defined as $$\xi{\rm xt}_{\mathcal{I}(\xi)}^n(M,N)=H^n({\Hom_\mathscr{C}}(M, \mathbf{I})).$$
\end{enumerate}}
\end{definition}

\begin{remark}\label{fact:2.5'} {
(1) In fact, there is  an isomorphism $\xi{\rm xt}_{\mathcal{P}(\xi)}^n(M,N)\cong \xi{\rm xt}_{\mathcal{I}(\xi)}^n(M,N),$
which is denoted by $\xi{\rm xt}_{\xi}^n(M,N)$ (see \cite[Definition 3.2]{HZZ1}).

(2)  Assume that $\mathscr{C}$ has enough $\xi$-projective objects. Using a standard argument in homological algebra, there is a bijection
$$
\xi{\rm xt}_\xi^1(M,N)\to \{[\xymatrix@C=0.4cm{N\ar[r]^{x}&Z\ar[r]^{y}& M}]\mid \xymatrix@C=0.4cm{N\ar[r]^{x}&Z\ar[r]^{y}& M\ar@{-->}[r]^{\delta}&}\in\xi\}.
$$
}
\end{remark}

\begin{remark} \label{remark-long} (\cite[Lemma 3.4]{HZZ1})
Let $\xymatrix@C=0.5cm{X\ar[r]&Y\ar[r]&Z\ar@{-->}[r]&}$
be an $\mathbb{E}$-triangle
in $\xi$.

 (1) If $\mathscr{C}$ has enough $\xi$-projective objects and $M$ is an object in $\mathscr{C}$,
then there exists a long exact sequence
\begin{center}
\xymatrix@C=0.5cm{0\ar[r]&\xi{\rm xt}_{{\xi}}^{0}(Z,M)\ar[r]&\xi {\rm xt}_{{\xi}}^{0}(Y,M)\ar[r]&\xi{\rm xt}_{{\xi}}^{0}(X,M)\ar[r]&}~~~~~~~~
\xymatrix@C=0.5cm{\xi{\rm xt}_{{\xi}}^{1}(Z,M)\ar[r]&\xi{\rm xt}_{{\xi}}^{1}(Y,M)\ar[r]&\xi{\rm xt}_{{\xi}}^{1}(X,M)
\ar[r]&\cdots}
\end{center}
 of abelian groups. Moreover, since $\xi$ is closed under cobase change, one has a long exact sequence
 \begin{center}
\xymatrix@C=0.5cm{{\rm Hom}_{\mathscr{C}}(Z,M)\ar[r]&{\rm Hom}_{\mathscr{C}}(Y,M)\ar[r]&{\rm Hom}_{\mathscr{C}}(X,M)\ar[r]&}~~~~~~~~
\xymatrix@C=0.5cm{\xi{\rm xt}_{{\xi}}^{1}(Z,M)\ar[r]&\xi{\rm xt}_{{\xi}}^{1}(Y,M)\ar[r]&\xi{\rm xt}_{{\xi}}^{1}(X,M)
\ar[r]&\cdots}
\end{center}
 of abelian groups by \cite[Corollary 3.12]{NP}.

 (2) Dual to (1).
\end{remark}

Following Remarks \ref{fact:2.5'} and  \ref{remark-long}, we have the following  results.

\begin{remark}\label{remark-ds}
Let $$\xymatrix@=0.5cm{
K \ar[r] &C_{n-1} \ar[r] &\cdots \ar[r] &C_{1} \ar[r] &C_{0} \ar[r] &A}$$ be a $\xi$-exact complex in the extriangulated category $\mathscr{C}$ such that $C_{i}\in {^{\bot}}Y,~\forall 0\leq i\leq n-1$.~Then $\xi xt^{k}_{\xi}(K,Y)\cong\xi xt^{k+n}_{\xi}(A,Y).$
\end{remark}

Now, we set
\begin{align*}
  \mathcal{X}^{\perp} & =\{M\in \mathscr{C}\mid \xi xt_{\xi}^{n\geq 1}(X,M)=0\text{ for all }X\in \mathcal{X}\}, \\
   ^{\perp}\mathcal{X}& =\{M\in \mathscr{C}\mid \xi xt_{\xi}^{n\geq 1}(M,X)=0\text{ for all }X\in \mathcal{X}\},\\
   \mathcal{X}^{\perp_k}&=\{M\in \mathscr{C}\mid \xi xt_{\xi}^{k}(X,M)=0\text{ for all }X\in \mathcal{X}\},\\
  ^{\perp_k}\mathcal{X} &=\{M\in \mathscr{C}\mid \xi xt_{\xi}^{k}(M,X)=0\text{ for all }X\in \mathcal{X}\}.
\end{align*}

For two subcategories $\mathcal{H}$ and $\mathcal{X}$ of $\mathscr{C}$, we sometimes write $\mathcal{H}\perp \mathcal{X}$ if $\mathcal{H}\subseteq {^{\perp}\mathcal{X}
}$ (equivalently,  $\mathcal{X}\subseteq \mathcal{H}^{\perp}$).

\begin{definition} {\rm (\cite[Definition 3.4]{HZZ})} Let
$\xymatrix@C=0.5cm{X \ar[r]^-{u}&Y\ar[r]^-{v}&Z\ar@{-->}[r]&}$ be an $\mathbb{E}$-triangle in $\xi$.
Then the morphism $u$ (resp. $v$) is called a \emph{{$\xi$-inflation}} (resp. a \emph{{$\xi$-deflation}}).
\end{definition}

Given any $\mathbb{E}$-triangle
$$\xymatrix@C=0.5cm{X \ar[r]& Y \ar[r]& Z \ar@{-->}[r]& }$$
in $\xi$.
We say that
$\mathcal{X}$ is \emph{{closed under $\xi$-extensions}} if, given any such $\mathbb{E}$-triangle in $\xi$ as above, the condition $X,Z\in\mathcal{X}$ implies $Y\in\mathcal{X}$.
We say that
$\mathcal{X}$ is \emph{closed under cocones of $\xi$-deflations} (resp. \emph{cones of $\xi$-inflations}) if, given any such $\mathbb{E}$-triangle in $\xi$ as above, the condition $Y,Z\in\mathcal{X}$ (resp. $X,Y\in\mathcal{X}$) implies $X\in\mathcal{X}$ (resp. $Z\in\mathcal{X}$).

Now we generalize the concept of resolving subcategories from abelian and triangulated categories to extriangulated categories.

\begin{definition}
{\rm Let $\mathscr{C}$ be an extriangulated category with enough $\xi$-projective objects and $\mathcal{X}$ a subcategory of $\mathscr{C}$. Then $\mathcal{X}$ is called a
{\it resolving} subcategory of $\mathscr{C}$ if the following conditions are satisfied.
\begin{enumerate}
\item[$(1)$] $\mathcal{P}(\xi)\subseteq \mathcal{X}$.
\item[$(2)$] $\mathcal{X}$ is closed under $\xi$-extensions.
\item[$(3)$] $\mathcal{X}$ is closed under cocones of $\xi$-deflations.
\end{enumerate}}

The notion of \emph{coresolving} subcategories is defined dually.
\end{definition}

\begin{remark}\label{remark}
\begin{itemize}
\item[(a)] We do not require that a resolving subcategory is closed under direct summands in the above definition.
\item[(b)] $\mathcal{P}(\xi)$ is a resolving subcategory.
\item[(c)] In \cite{HZZ}, Hu et al. introduced the notion of  $\xi$-$\mathcal{G}$projective (resp., $\xi$-$\mathcal{G}$injective) objects, see \cite[Definion 4.8]{HZZ} for the definition.   We denote by $\mathcal{GP}(\xi)$ (resp. $\mathcal{GI}(\xi)$) the class of $\xi$-$\mathcal{G}$projective (resp., $\xi$-$\mathcal{G}$injective) objects. Then $\mathcal{GP}(\xi)$ is a resolving subcategory by \cite[Theorems 4.16 and 4.17]{HZZ}.
\end{itemize}
\end{remark}

\begin{definition}(\cite[Definition 4.1]{{HZZ1}}) A subcategory $\mathcal{X}$ of $\mathscr{C}$
is called  a {\it generating subcategory}
of $\mathscr{C}$ if for all $X\in \mathcal{X}$, the condition
$\Hom_{\mathscr{C}}(X,C)=0$ implies $C=0$. Dually,
a subcategory $\mathcal{Y}$ is called a {\it cogenerating subcategory}
of $\mathscr{C}$ if  for all $Y\in \mathcal{Y}$, the condition
$\Hom_{\mathscr{C}}(C,Y)=0$ implies $C=0$.
\end{definition}

\begin{definition}(cf. \cite[Definition 3.8]{MZResolving})
Let $\mathcal{X}$ be a subcategory of $\mathscr{C}$ and $M$ an object in $\mathscr{C}$.
A $\xi$-deflation $\xymatrix@C=0.5cm{X\ar[r]& M}$ is said to be a
\emph{right $\mathcal{X}$-approximation} of $M$ if ${X}\in \mathcal{X}$ and the induced complex
$$\xymatrix@C=0.5cm{\Hom_{\mathscr{C}}(\widetilde{X},X)\ar[r]
& \Hom_{\mathscr{C}}(\widetilde{X},M)\ar[r]&0}$$ is exact for any
$\widetilde{X}\in \mathcal{X}$. In this case, there is an $\mathbb{E}$-triangle
$\xymatrix@C=0.5cm{K\ar[r]&X\ar[r]&M\ar@{-->}[r]&}$ in $\xi$.
Dually, a \emph{left $\mathcal{X}$-approximation} of $M$ is defined.
\end{definition}

Now we generalize the concepts of contravariantly finite, covariantly finite and functorially finite subcategories from abelian categories to extriangulated categories.

We say that the subcategory $\mathcal{X}$ is \emph{contravariantly finite}
if any object $M\in \mathscr{C}$ admits a right $\mathcal{X}$-approximation.
Dually, we say that $\mathcal{X}$ is {\it covariantly finite} if any object
$M\in \mathscr{C}$ admits a left $\mathcal{X}$-approximation.
The subcategory $\mathcal{X}$ is called \emph{functorially finite}
if it is both contravariantly finite and covariantly finite.

\begin{definition}{\rm (cf. \cite[Definition 2.11]{MZResolving})}
{\rm Let $(\mathcal{X},\omega)$ be a pair of subcategories in $\mathscr{C}$ with $\omega\subseteq \mathcal{X}$.
\begin{itemize}
\item[(1)] $\omega$ is called a {\it $\xi$-cogenerator} of $\mathcal{X}$ if for any  $X\in\mathcal{X}$,
there exists an $\mathbb{E}$-triangle
$$\xymatrix@C=0.5cm{X \ar[r]&W\ar[r]&X'\ar@{-->}[r]&}$$ in $\xi$ with $W\in\omega$ and $X'\in\mathcal{X}$.
\item[(2)] $\omega$ is called  \emph{$\mathcal{X}$-injective} if $\omega\subseteq \mathcal{X}^{\perp}$.
\end{itemize}}
\end{definition}

\begin{definition}
{\rm Let $\mathcal{X}$ be a subcategory of $\mathscr{C}$ and $M\in\mathscr{C}$.
The {\it $\mathcal{X}$-resolution dimension} of $M$ (with respect to $\xi$), written $\mathcal{X}$-$\dim M$, is defined by
\begin{align*}
\mathcal{X}\text{-}\dim M&=\inf \{n \geq 0\mid\text{ there exists a } \xi\text{-exact complex}\\
&\xymatrix@C=0.5cm
{X_{n}\ar[r]&\cdots\ar[r]&X_{1}\ar[r]&X_{0}\ar[r]&M} \text{ in } \mathscr{C} \text{ with all } X_{i} \text{ objects in } \mathcal{X}\}.
\end{align*}
The \emph{$\mathcal{X}$-resolution dimension} of $\mathscr{C}$ is defined by
\begin{align*}
\mathcal{X}\text{-}\dim \mathscr{C}:=&\sup \{\mathcal{X}\text{-}\res M \mid M\in \mathscr{C}\}.
\end{align*}
Dually, the \emph{$\mathcal{X}$-coresolution dimensions} $\mathcal{X}\text{-}\cores M$ and $\mathcal{X}\text{-}\cores \mathscr{C}$ are defined.

For a $\xi$-exact complex
$$\xymatrix@C=0.5cm
{\cdots\ar[r]^{f_{n+1}}&X_{n}\ar[r]&\cdots\ar[r]^{f_{2}}&X_{1}\ar[r]^{f_{1}}&X_{0}\ar[r]^{f_{0}}&M}$$ with all $X_{i}\in \mathcal{X}$, there are
$\mathbb{E}$-triangles $\xymatrix@C=0.5cm{K_1\ar[r]^{g_0}&X_0\ar[r]^{f_0}&M\ar@{-->}[r]&}$ and  $\xymatrix@C=0.5cm{K_{i+1}\ar[r]^{g_i}&X_i\ar[r]^{h_i}&K_i\ar@{-->}[r]&}$ with $f_i=g_{i-1}h_i$ for each $i>0$.
The object $K_i$ is called an $i$th {\it $\mathcal{X}$-syzygy} of $M$, denoted by $\Omega^{i}_{\mathcal{X}}(M)$. In case  $\mathcal{X}=\mathcal{P}(\xi)$, we have $\xi\text{-}\pd M=\mathcal{X}\text{-}\dim M$ and write $\Omega^{i}(M):=\Omega^{i}_{\mathcal{P}(\xi)}(M)$.
In case  $\mathcal{X}=\mathcal{GP}(\xi)$, $\mathcal{X}\text{-}\dim M$ coincides with $\xi\text{-}\mathcal{G}\pd M$ defined by Hu, Zhang and Zhou \cite{HZZ} as $\xi$-$\mathcal{G}$projective dimension, the proof is straightforward. Dually, the $\xi\text{-}\id M \mbox{ and }\xi\text{-}\mathcal{G}\id M$ are defined.}
\end{definition}

In \cite{HZZ1}, Hu et al. introduced the following condition:

 {\bf Condition(*)}\label{condition*}
If $X\in\mathscr{C}$ and $Y$ is an object in $\mathscr{C}$ with finite $\xi\text{-}\mathcal{G}$projective dimension such that $\xi xt^{i}_{\xi}(Y,X)=0$ for any $i>0$, then $\Hom_{\mathscr{C}}(Y,X)\cong\xi xt^{0}_{\xi}(Y,X).$ Dually, if $X$, $Y$ are objects in $\mathscr{C}$ with finite $\xi\text{-}\mathcal{G}$injective dimension such that $\xi xt^{i}_{\xi}(X,Y)=0$ for any $i>0$, then $\Hom_{\mathscr{C}}(X,Y)\cong\xi xt^{0}_{\xi}(X,Y).$

\begin{remark}\label{2.29}{\rm (\cite[Proposition 4.6 and Corollary 4.8]{HZZ1})} Let $M\in\mathscr{C}$ with finite $\xi\text{-}\mathcal{G}$projective dimension and finite $\xi\text{-}\mathcal{G}$injective dimension respectively. Then
\begin{itemize}

\item[(1)] $\xi\text{-}\mathcal{G}\pd M=\xi\text{-}\pd M.$
\item[(2)] $\xi\text{-}\mathcal{G}\id M=\xi\text{-}\id M.$
\end{itemize}

In this case, if $\mathcal{P}(\xi)$ be a generating subcategory of $\mathscr{C}$ and $\mathcal{I}(\xi)$ is a cogenerating  subcategory of $\mathscr{C}$ and $\mathscr{C}$ satisfies condition(*). Then we have:
\begin{itemize}
\item[(3)] $\sup\{\xi\mbox{-}\mathcal{G}\pd X\mid for~any ~X\in\mathscr{C}\}=\sup\{\xi\mbox{-}\mathcal{G}\id X\mid for~any~X\in\mathscr{C}\}$.
\end{itemize}
\end{remark}

We use ${\mathcal{X}}^{\wedge}$ (resp. ${\mathcal{X}}^{\vee}$) to denote the subcategory
of $\mathscr{C}$ consisting of objects having finite $\mathcal{X}$-resolution (resp. $\mathcal{X}$-coresolution)
dimension, and use ${\mathcal{X}}_{n}^{\wedge}$ (resp. ${\mathcal{X}}_{n}^{\vee}$)
to denote the subcategory of $\mathscr{C}$ consisting of
objects having $\mathcal{X}$-resolution dimension (resp. $\mathcal{X}$-coresolution) at most $n$.

{\it In the following sections, we always assume that  $\mathscr{C}=(\mathscr{C}, \mathbb{E}, \mathfrak{s})$ is an extriangulated category and $\xi$ is a proper class of $\mathbb{E}$-triangles in  $\mathscr{C}$.  We also assume that the extriangulated category $\mathscr{C}$ has enough $\xi$-projectives and enough $\xi$-injectives satisfying Condition (WIC).}

\section{Left $n$-cotorsion pairs}

We first recall the notion of left (resp. right) $n$-cotorsion pairs in extriangulated categories with respect to a proper class of triangles $\xi$. Then we
investigate the properties of left $n$-cotorsion pairs and discuss the relation between $n$-cotorsion pairs and cotorsion pairs.

In what follows, we always assume that $n$ is a positive integer.

\begin{definition}(\cite[Definition 3.1]{ZPY})
Let $\mathcal{U}$ and $\mathcal{V}$ be two subcategories of $\mathscr{C}$.
We say that $(\mathcal{U},\mathcal{V})$ is a \emph{left $n$-cotorsion
pair} in $\mathscr{C}$ if the following conditions are satisfied.
\begin{itemize}
\item[(LN1)] $\mathcal{U}$ is closed under direct summands.
\item[(LN2)] $\xi xt_{\xi}^{1 \leq i \leq n}(\mathcal{U},\mathcal{V})=0$.
\item[(LN3)] Every object $M\in \mathscr{C}$ admits an $\mathbb{E}$-triangle
$$\xymatrix@C=0.5cm{K\ar[r]&U\ar[r]&M\ar@{-->}[r]&}$$
in $\xi$ with $U\in \mathcal{U}$ and $K\in{\mathcal{V}}^{\wedge}_{n-1}$.
\end{itemize}

Dually, we say that $(\mathcal{U},\mathcal{V})$ is a  \emph{right $n$-cotorsion pair}
in $\mathscr{C}$ if the following conditions are satisfied.
\begin{itemize}
\item[(RN1)] $\mathcal{V}$ is closed under direct summands.
\item[(RN2)] $\xi xt_{\xi}^{1 \leq i \leq n}(\mathcal{U},\mathcal{V})=0$.
\item[(RN3)] Every object $N\in \mathscr{C}$ admits an $\mathbb{E}$-triangle
$$\xymatrix@C=0.5cm{N\ar[r]&V'\ar[r]&K'\ar@{-->}[r]&}$$
in $\xi$ with $V'\in \mathcal{V}$ and $K'\in{\mathcal{U}}^{\vee}_{n-1}$.
\end{itemize}

We say that $(\mathcal{U},\mathcal{V})$ is an \emph{$n$-cotorsion pair} in $\mathscr{C}$ if
$(\mathcal{U},\mathcal{V})$ is both a left and right $n$-cotorsion pair in $\mathscr{C}$.

In particular, left (resp. right) $1$-cotorsion pairs are called
\emph{left (resp. right) cotorsion pairs}, and $1$-cotorsion pairs are called
 \emph{cotorsion pairs}.
 \end{definition}

 \begin{remark}\label{remark-contra}
Let $\mathcal{U}$ and $\mathcal{V}$ be subcategories  of $\mathscr{C}$.
\begin{itemize}
\item[(1)] If $(\mathcal{U},\mathcal{V})$ is a left cotorsion pair in $\mathscr{C}$, then
$\mathcal{U}={^{\perp_{1}} \mathcal{V}}$. Moreover, we have that $\mathcal{P}(\xi)\subseteq \mathcal{U}$,
$\mathcal{U}$ is closed under $\xi$-extensions, and
$\mathcal{U}$ is a contravariantly finite subcategory of $\mathscr{C}$.
\end{itemize}

\begin{itemize}
\item[(2)] If $(\mathcal{U},\mathcal{V})$ is a right cotorsion pair in $\mathscr{C}$, then
$\mathcal{V}= \mathcal{U}^{\perp_{1}}$.  Moreover, we have that $\mathcal{I}(\xi)\subseteq \mathcal{V}$,
$\mathcal{V}$ is closed under $\xi$-extensions, and $\mathcal{V}$ is a convariantly finite subcategory
of $\mathscr{C}$.\\
\end{itemize}
\end{remark}

Now we see an example from \cite{ZPY}. We denote by ``$\circ$" in the Auslander-Reiten quiver the indecomposable objects which belong to the subcategory $\mathcal{X}$.
Let $\Lambda$ be the algebra given by the following quiver with relations:
$$
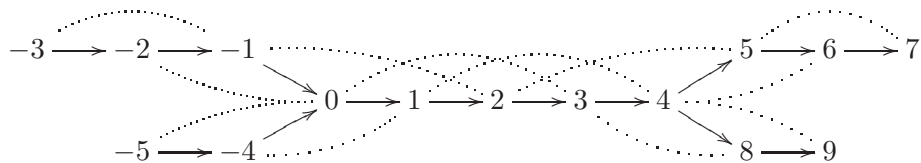
\xymatrix@C=0.7cm@R0.2cm{
&\\
-3 \ar[r] \ar@{.}@/^15pt/[rr] &-2 \ar@{.}@/_6pt/[drr]\ar[r] &-1 \ar@{.}@/^8pt/[drrr]\ar[dr] &&&&&&5 \ar[r] \ar@{.}@/^15pt/[rr] &6 \ar[r] &7\\
&&&0 \ar[r] \ar@{.}@/^18pt/[rrr] & 1 \ar[r] \ar@{.}@/^18pt/[rrr] &2 \ar[r] \ar@{.}@/^8pt/[urrr] & 3 \ar[r] \ar@{.}@/_8pt/[drr] &4 \ar[ur] \ar[dr] \ar@{.}@/_6pt/[urr] \ar@{.}@/^6pt/[drr] \\
&-5 \ar[r] \ar@{.}@/^6pt/[urr] &-4 \ar[ur] \ar@{.}@/_8pt/[urr] &&&&&&8 \ar[r] &9
}
$$
There exists a $3$-cluster tilting subcategory $\mathcal{X}$ of~~$\mathscr{C}={\rm mod}\Lambda$ (\cite[Example 4.21]{V})  as follows:
$$\xymatrix@C=0.2cm@R0.2cm{
&&&\circ \ar[dr] &&&&\circ \ar[dr] &&\circ \ar[dr] &&\circ \ar[dr] &&\circ \ar[dr] &&\circ \ar[dr] &&&&\circ \ar[dr]\\
&\mathcal{X}: &\circ \ar@{.}[rr] \ar[ur] &&\cdot \ar[dr] \ar@{.}[rr] &&\cdot \ar[dr] \ar[ur] \ar@{.}[rr] &&\cdot \ar[dr] \ar[ur] \ar@{.}[rr] &&\cdot \ar[dr] \ar[ur] \ar@{.}[rr] &&\cdot \ar[dr] \ar[ur] \ar@{.}[rr] &&\cdot \ar[dr] \ar[ur] \ar@{.}[rr] &&\cdot \ar[dr] \ar@{.}[rr] &&\cdot  \ar[ur] \ar@{.}[rr] &&\circ\\
&&&&&\circ \ar[dr] \ar[ur] \ar@{.}[rr] &&\cdot \ar[ur] \ar@{.}[rr] &&\circ \ar[ur] \ar@{.}[rr] &&\cdot \ar[ur] \ar@{.}[rr] &&\circ \ar[ur] \ar@{.}[rr] &&\cdot \ar[dr] \ar[ur] \ar@{.}[rr] &&\circ \ar[dr] \ar[ur]\\
\circ \ar@{.}[rr] \ar[dr] &&\cdot \ar@{.}[rr] \ar[dr] &&\cdot \ar[ur] \ar@{.}[rr] &&\circ &&&&&&&&&&\circ \ar[ur] \ar@{.}[rr] &&\cdot \ar[dr] \ar@{.}[rr] &&\cdot \ar[dr] \ar@{.}[rr] &&\circ\\
&\circ \ar[ur] &&\circ \ar[ur] &&&&&&&&&&&&&&&&\circ \ar[ur] && \circ \ar[ur]
}
$$
By \cite[Example 3.7]{ZPY},  $(\mathcal{X},\mathcal{X})$ is a $2$-cotorsion pair in $\mathscr{C}$.  See \cite{ZPY} for more examples.

\begin{proposition}\label{prop-1}
Let $\mathcal{U}$ and $\mathcal{V}$ be subcategories of $\mathscr{C}$ with ${\mathcal{V}}\subseteq\bigcap\limits_{i=1}^{n}\mathcal{U}^{\perp_i}$.
 If
$Y\in {\mathcal{V}}^{\wedge}_{k}$ for some $k$ with $0\leq k \leq n-1$, then $Y\in \bigcap\limits_{i=1}^{n-k}\mathcal{U}^{\perp_i}$.
In particular, ${{\mathcal{V}}^{\wedge}_{n-1}}\subseteq\mathcal{U}^{\perp_1}$.
\end{proposition}

\begin{proof}
The case $n=1$ is trivial. Now suppose $n\geq 2$.
The $k=0$ is trivial, so we suppose $1\leq k\leq n-1$.
Let $U\in\mathcal{U}$ and $Y\in {\mathcal{V}}^{\wedge}_{k}$. We will proceed by induction on $k$. For  $k=1$, there is an $\mathbb{E}$-triangle
$\xymatrix@C=0.5cm{V_{1}\ar[r]&V_{0}\ar[r]&Y\ar@{-->}[r]&}$
in $\xi$ with $V_{1},\ V_{0}\in \mathcal{V}$.
Applying the functor $\Hom_{\mathscr{C}}(U,-)$ to the above $\mathbb{E}$-triangle yields the following exact sequence
$$\xymatrix@C=0.5cm{\cdots\ar[r]&\xi xt_{\xi}^{i}(U,V_{0})\ar[r]&
\xi xt_{\xi}^{i}(U,Y)\ar[r]&\xi xt_{\xi}^{i+1}(U,V_{1})\ar[r]&\cdots}.$$
For any $1\leq i\leq n-1$,
since $\xi xt_{\xi}^{i}(U,V_{0})=0=\xi xt_{\xi}^{i+1}(U,V_{1})$, we have $\xi xt_{\xi}^{i}(U,Y)=0$.

Now suppose $2\leq k\leq n-1$.
there is an $\mathbb{E}$-triangle
$\xymatrix@C=0.5cm{Y'\ar[r]&V'_{0}\ar[r]&Y\ar@{-->}[r]&}$
in $\xi$ with $Y'\in\mathcal{V}^{\wedge}_{k-1}$ and $V'_{0}\in \mathcal{V}$.
Applying the functor $\Hom_{\mathscr{C}}(U,-)$ to the above $\mathbb{E}$-triangle yields the following exact sequence
$$\xymatrix@C=0.5cm{\cdots\ar[r]&\xi xt_{\xi}^{i}(U,V'_{0})\ar[r]&
\xi xt_{\xi}^{i}(U,Y)\ar[r]&\xi xt_{\xi}^{i+1}(U,Y')\ar[r]&\cdots}.$$
Since $\xi xt_{\xi}^{1\leq i\leq n-k}(U,V'_{0})=0$ by assumption, and $\xi xt_{\xi}^{1\leq i\leq n-k+1}(U,Y')=0$ by the induction hypothesis,
we have $\xi xt_{\xi}^{1\leq i\leq n-k}(U,Y)=0$.
\end{proof}

The following result gives an equivalent characterization of left $n$-cotorsion pairs.

\begin{lemma}{\rm (cf. \cite[Lemma 3.4]{ZPY})}\label{thm-1}
Let $\mathcal{U}$ and $\mathcal{V}$ be subcategories of $\mathscr{C}$.
Then the following statements are equivalent:
\begin{itemize}
\item[(1)] $(\mathcal{U},\mathcal{V})$ is a left $n$-cotorsion pair in $\mathscr{C}$.
\item[(2)] $\mathcal{U}=\bigcap\limits_{i=1}^{n}$$^{\perp_{i}}{{\mathcal{V}}}$
and for any object $T\in \mathscr{C}$, there is an $\mathbb{E}$-triangle
$$\xymatrix@C=0.5cm{K\ar[r]&U\ar[r]&T\ar@{-->}[r]&}$$
in $\mathcal{\xi}$ with $U\in \mathcal{U}$ and $K\in {\mathcal{V}}^{\wedge}_{n-1}$.
\end{itemize}
Moreover, if one of the above conditions holds true, then $(\mathcal{U},{\mathcal{V}}^{\wedge}_{n-1})$
is a left cotorsion pair in $\mathscr{C}$.
\end{lemma}

By   the dual of Lemma \ref{thm-1},  if $(\mathcal{U},\mathcal{V})$ is a right $n$-cotorsion pair in $\mathscr{C}$, then
$({\mathcal{U}}^{\vee}_{n-1},{\mathcal{V}})$
is a right cotorsion pair in $\mathscr{C}$.

In the rest of this section, we give some properties related to (left) $n$-cotorsion pairs.

\begin{proposition}\label{prop-3.7}
Assume that $(\mathcal{U},\mathcal{V})$ is an $n$-cotorsion pair in $\mathscr{C}$.
Then the following statements are equivalent.
\begin{itemize}
\item[(1)] $\mathcal{U}\subseteq \mathcal{V}$.
\item[(2)] $\mathscr{C}={\mathcal{V}}^{\wedge}_{n}$.
\item[(3)] ${\mathcal{U}}^{\vee}_{n-1}\subseteq{^{\perp_1}\mathcal{U}}$.
\end{itemize}
\end{proposition}

\begin{proof}
$(1)\Longrightarrow(2)$ Clearly.

$(2)\Longrightarrow(1)$ Let $U\in \mathcal{U}\subseteq \mathscr{C}$. By assumption, there is an $\mathbb{E}$-triangle
$\xymatrix@C=0.5cm{K\ar[r]&V_{0}\ar[r]&U\ar@{-->}[r]&}$
in $\xi$ with $K\in {\mathcal{V}}^{\wedge}_{n-1}$ and $V_{0}\in \mathcal{V}$.
By Lemma \ref{thm-1}, $(\mathcal{U},{\mathcal{V}}^{\wedge}_{n-1})$
is a left cotorsion pair in $\mathscr{C}$. In particular, $\xi xt^1_{\xi}(\mathcal{U},{\mathcal{V}}^{\wedge}_{n-1})=0$, which shows that the above $\mathbb{E}$-triangle is split. Thus  $U\in \mathcal{V}$ as a direct summand of $V_{0}$. Therefore, $\mathcal{U}\subseteq \mathcal{V}$.

$(1)\Longleftrightarrow(3)$ It follows from the dual of Lemma \ref{thm-1}.
\end{proof}

It is clear that ${\mathcal{V}}^{\wedge}=\mathcal{V}$ if $\mathcal{V}$ is a coresolving subcategory,
and ${\mathcal{U}}^{\vee}=\mathcal{U}$ if $\mathcal{U}$ is a resolving subcategory.
By Proposition \ref{prop-3.7} and Lemma \ref{thm-1}, we have the following result.

\begin{proposition}\label{prop-left-resolving}
Let $(\mathcal{U},\mathcal{V})$ be an $n$-cotorsion pair in $\mathscr{C}$ with
$\mathcal{U}\subseteq{^{\perp_{n+1}}\mathcal{V}}$. Then  the following statements are equivalent.
\begin{itemize}
\item[(1)] $\mathcal{U}\subseteq \mathcal{V}$.
\item[(2)] $\mathscr{C}=\mathcal{V}$.
\item[(3)] $\mathcal{U}\subseteq{^{\perp_1}\mathcal{U}}$.
\end{itemize}
\end{proposition}

\begin{proof}
By Proposition \ref{prop-3.7}, it suffices to show that $\mathcal{U}$ is resolving. By Lemma \ref{thm-1}, $(\mathcal{U},\mathcal{V}^{\wedge}_{n-1})$ is a left cotorsion pair in $\mathscr{C}$.
By Remark \ref{remark-contra},  $\mathcal{U}$ is closed under $\xi$-extensions and
$\mathcal{P}(\xi)\subseteq \mathcal{U}$. Now, given an $\mathbb{E}$-triangle
$\xymatrix@C=0.5cm{U\ar[r]&U'\ar[r]&U''\ar@{-->}[r]&}$
 in $\xi$ with $U',\ U''\in \mathcal{U}$.
For any $V\in \mathcal{V}$, applying the functor $\Hom_{\mathscr{C}}(-,V)$
to the above $\mathbb{E}$-triangle yields the following exact sequence
$$\xymatrix@C=0.5cm{\cdots\ar[r]&\xi xt^{i}_{\xi}(U',V)\ar[r]
&\xi xt^{i}_{\xi}(U,V)\ar[r]&\xi xt^{i+1}_{\xi}(U'',V)\ar[r]&\cdots}.$$
By assumption $\xi xt_{\xi}^{1\leq i\leq n+1}(\mathcal{U},\mathcal{V})=0$, this implies $\xi xt_{\xi}^{1\leq i\leq n}(U,V)=0$.
Thus $U\in \bigcap\limits_{i=1}^{n}{^{\perp_{i}}\mathcal{V}}=\mathcal{U}$ by Lemma \ref{thm-1},
and hence $\mathcal{U}$ is closed under cocones of $\xi$-deflations. Therefore,
$\mathcal{U}$ is resolving.
\end{proof}

The following result establishes a relation between $n$-cotorsion pairs and cotorsion pairs.

\begin{proposition}\label{prop-6}
Let $\mathcal{U}$ and $\mathcal{V}$ be subcategories in $\mathscr{C}$. Then the following statements are equivalent.
\begin{itemize}
\item[(1)] $(\mathcal{U},\mathcal{V})$ is an $n$-cotorsion pair in $\mathscr{C}$ with
$\mathcal{U}\subseteq{^{\perp_{n+1}}\mathcal{V}}$.
\item[(2)] $(\mathcal{U},\mathcal{V})$ is an $n$-cotorsion pair
in $\mathscr{C}$ with $\mathcal{U}$ a resolving subcategory.
\item[(3)] $(\mathcal{U},\mathcal{V})$ is an $n$-cotorsion pair
in $\mathscr{C}$ with $\mathcal{V}$ a coresolving subcategory.
\item[(4)] $(\mathcal{U},\mathcal{V})$ is a cotorsion pair
in $\mathscr{C}$ with $\mathcal{U}$ a resolving subcategory.
\item[(5)] $(\mathcal{U},\mathcal{V})$ is a cotorsion pair
in $\mathscr{C}$ with $\mathcal{V}$ a coresolving subcategory.
\end{itemize}
Moreover, if one of the above conditions holds true, then
$\mathcal{U}\subseteq{^{\perp}\mathcal{V}}$.
\end{proposition}

\begin{proof}
$(1)\Longrightarrow(2)$ It follows from Proposition \ref{prop-left-resolving}.

$(2)\Longrightarrow(1)$
Let $U\in \mathcal{U}$ and $V\in \mathcal{V}$.
Consider an   $\mathbb{E}$-triangle
$\xymatrix@C=0.5cm{U'\ar[r]&P\ar[r]&U\ar@{-->}[r]&}$
in $\xi$ with $P\in \mathcal{P}(\xi)$.
Since $\mathcal{U}$ is resolving, then $U'\in \mathcal{U}$.
By the dimension shifting,
we have $\xi xt_{\xi}^{n+1}(\mathcal{U},\mathcal{V})=0$.

$(1)\Longleftrightarrow(3)$ It is a dual of $(1)\Longleftrightarrow(2)$.

$(2)\Longrightarrow(4)$ or $(3)\Longrightarrow(4)$
{By Lemma \ref{thm-1}, $(\mathcal{U},\mathcal{V}^{\wedge}_{n-1})$
is a left cotorsion pair in $\mathscr{C}$. Since $\mathcal{V}$ is coresolving,
$\mathcal{V}=\mathcal{V}^{\wedge}_{n-1}$, then $(\mathcal{U},\mathcal{V})$
is a left cotorsion pair in $\mathscr{C}$. Dually, $(\mathcal{U},\mathcal{V})$
is a right cotorsion pair in $\mathscr{C}$. Thus $(\mathcal{U},\mathcal{V})$
is a cotorsion pair in $\mathscr{C}$, and $\mathcal{U}$ is resolving.}

$(4)\Longrightarrow(2)$ By using an argument similar to
that of the implication $(2)\Longrightarrow(1)$,
we get $\xi xt_{\xi}^{1\leq i\leq n}(\mathcal{U},\mathcal{V})=0$.

$(3)\Longrightarrow(5)$ By the dual of Lemma 3.6.

$(5)\Longrightarrow(3)$ By  the dual of $(4)\Longrightarrow(2)$.
\end{proof}

By Proposition \ref{prop-6}, we immediately have the following result.

\begin{corollary}\label{lemma-resolving}
Let  $(\mathcal{U},\mathcal{V})$ be a cotorsion pair in $\mathscr{C}$.
Then the following statements are equivalent.
\begin{itemize}
\item[(1)] $\mathcal{U}\subseteq{^{\perp_2}\mathcal{V}}$.
\item[(2)] $\mathcal{U}$ is a resolving subcategory.
\item[(3)] $\mathcal{V}$ is a coresolving subcategory.
\end{itemize}
Moreover, if one of the above conditions holds true, then
$\mathcal{U}\subseteq{^{\perp}\mathcal{V}}$.
\end{corollary}

\begin{lemma}\label{lemma-2}
Let $\mathcal{U}$ and $\mathcal{V}$ be subcategories of $\mathscr{C}$
such that $\mathcal{U}\subseteq\bigcap\limits_{i=1}^{n}$$^{\perp_{i}}{{\mathcal{V}}}$. Then
${\mathcal{U}}^{\wedge}_{k}\subseteq {^{\perp_{k+1}}}\mathcal{V}$ for any $0\leq k\leq n-1$.
\end{lemma}

\begin{proof}
The case $k=0$ is trivial.
Let $X\in {\mathcal{U}}^{\wedge}_{k}$ and $V\in \mathcal{V}$. We will proceed by induction on $k$.
For  $k=1$, there is an $\mathbb{E}$-triangle
$\xymatrix@C=0.5cm{U_{1}\ar[r]&U_{0}\ar[r]&X\ar@{-->}[r]&}$
in $\xi$ with $U_{1},\ U_{0}\in{\mathcal{U}}$.
By the dimension shifting, we have
$\xi xt^{2}_{\xi}(X,V)=0$, and hence $X\in{^{\perp_{2}}}\mathcal{V}$.

Now suppose $2\leq k\leq n-1$. There is an $\mathbb{E}$-triangle
$\xymatrix@C=0.5cm{K\ar[r]&U'_{0}\ar[r]&X\ar@{-->}[r]&}$
in $\xi$ with $U'_{0}\in{\mathcal{U}}$ and $K\in {\mathcal{U}}^{\wedge}_{k-1}$.
Applying the functor $\Hom_\mathscr{C}(-,V)$ to the above $\mathbb{E}$-triangle yields the following exact sequence
$$\xymatrix@C=0.5cm{\cdots\ar[r]&\xi xt^{k}_{\xi}(K,V)
\ar[r]&\xi xt^{k+1}_{\xi}(X,V)\ar[r]&\xi xt^{k+1}_{\xi}(U'_{0},V)\ar[r]&\cdots}.$$
Since $\xi xt^{k}_{\xi}(K,V)=0$ by the induction hypothesis, and since $\xi xt^{k+1}_{\xi}(U'_{0},V)=0$ by assumption,
we have $\xi xt^{k+1}_{\xi}(X,V)=0$ and $X\in{^{\perp_{k+1}}}\mathcal{V}$.
Thus $\mathcal{U}^{\wedge}_{k}\subseteq {^{\perp_{k+1}}\mathcal{V}}$ for any $0\leq k\leq n-1$.
\end{proof}

As a consequence, we get the following proposition.

\begin{proposition}\label{prop-7}
Let $(\mathcal{U},\mathcal{V})$ be a left $n$-cotorsion pair in $\mathscr{C}$.
Then the following statements are equivalent.
\begin{itemize}
\item[(1)] $\mathcal{U}={^{\perp_{1}}\mathcal{V}}$.
\item[(2)] ${\mathcal{U}}^{\wedge}_{k}= {^{\perp_{k+1}}}\mathcal{V}$ for any $0\leq k\leq n-1$.
\end{itemize}
\end{proposition}

\begin{proof}
$(2)\Longrightarrow(1)$ It is trivial by setting $k=0$ in (2).

$(1)\Longrightarrow(2)$ The case $k=0$ is clear.
Now assume $k\geq 1$. By Lemma \ref{lemma-2},
${\mathcal{U}}^{\wedge}_{k}\subseteq {^{\perp_{k+1}}}\mathcal{V}$.
Conversely, let $Y\in{^{\perp_{k+1}}}\mathcal{V}$.
Consider the following $\mathbb{E}$-triangle
$\xymatrix@C=0.5cm{K_{1}\ar[r]&U_{0}\ar[r]&Y\ar@{-->}[r]&}$
in $\xi$ with $U_{0}\in{\mathcal{U}}$ and $K_{1}\in {\mathcal{V}}^{\wedge}_{n-1}$.
Repeating this process, we get the following $\xi$-exact complex
$$\xymatrix@C=0.5cm{K_{k}\ar[r]&U_{k-1}\ar[r]&\cdots\ar[r]&U_{1}\ar[r]&U_{0}\ar[r]&Y}$$
with $U_{i}\in \mathcal{U}$ for $0\leq i\leq k-1$. Applying the functor $\Hom_{\mathscr{C}}(-,V)$ to it, $\xi xt^{i\geq1}_{\xi}(\mathcal{U},V)=0$ since $(\mathcal{U},\mathcal{V})$ is a left $n$-cotorsion pair. By Remark \ref{remark-ds},
$\xi xt^{1}_{\xi}(K_{k},V)\cong\xi xt^{k+1}_{\xi}(Y,V)=0$, which implies
$K_{k}\in {^{\perp_{1}}}\mathcal{V}=\mathcal{U}$ by assumption.
Hence $Y\in {\mathcal{U}}^{\wedge}_{k}$ and ${^{\perp_{k+1}}}\mathcal{V}\subseteq{\mathcal{U}}^{\wedge}_{k}$.
Thus ${\mathcal{U}}^{\wedge}_{k}={^{\perp_{k+1}}}\mathcal{V}$.
\end{proof}

\begin{corollary}\label{cor-2}
Let
$(\mathcal{U},\mathcal{V})$ be a left $n$-cotorsion pair in $\mathscr{C}$.
If $\mathcal{U}={^{\perp_{1}}\mathcal{V}}$, then for any $0\leq k \leq n-1$,
the following statements are equivalent.
\begin{itemize}
\item[(1)] $\mathcal{U}\text{-}\res \mathscr{C}\leq k$.
\item[(2)] $\mathscr{C}={^{\perp_{k+1}}\mathcal{V}}$.
\end{itemize}
\end{corollary}

By Propositions \ref{prop-6} and \ref{prop-7}, we have
the following result.

\begin{corollary}
Let $(\mathcal{U},\mathcal{V})$ be a cotorsion pair with $\mathcal{U}$  a resolving subcategory. Then
for any nonnegative integers $m$ and $n$, we have
$\mathcal{U}^{\perp_{m+1}}={\mathcal{V}}^{\vee}_{m}$ and
${^{\perp_{n+1}}\mathcal{V}}={\mathcal{U}}^{\wedge}_{n}$.
\end{corollary}

\section{Left Frobenius pairs and weak Auslander-Buchweitz contexts}

In this section, we will introduce the notions of left Frobenius pairs and left (weak) Auslander-Buchweitz contexts and give a one-one correspondence between left Frobenius pairs and left (weak) Auslander-Buchweitz contexts.

We begin with some relations on special subcategories.

\begin{proposition}\label{prop-injective-1}
Let $\mathcal{X}$ and $\omega$ be  subcategories of $\mathscr{C}$. Assume that
 $\omega$ is $\mathcal{X}$-injective.
\begin{itemize}
\item[(1)] $\omega^{\wedge}\subseteq\mathcal{X}{^\perp}$.
\item[(2)] Assume that $\omega$ is a $\xi$-cogenerator for $\mathcal{X}$ and
$\omega$ is closed under direct summands in $\mathscr{C}$. Then
$$\omega=\mathcal{X}\cap\omega^{\wedge}=\mathcal{X}\cap\mathcal{X}^{\perp}.$$
\end{itemize}
\end{proposition}

\begin{proof}
(1) Easily.

(2) By (1), $\omega\subseteq\mathcal{X}\cap\omega^{\wedge}\subseteq\mathcal{X}\cap\mathcal{X}^{\perp}$.
Next we prove that  $\mathcal{X}\cap\mathcal{X}^{\perp}\subseteq\omega$.
To do it, let $X\in \mathcal{X}\cap\mathcal{X}^{\perp}$.
By assumption,  $\omega$ is a $\xi$-cogenerator in $\mathcal{X}$. Then there is an $\mathbb{E}$-triangle
$\xymatrix@C=0.5cm{X\ar[r]&W\ar[r]&X'\ar@{-->}[r]&}$
in $\xi$ with $W\in \omega$ and $X'\in \mathcal{X}$.
Since $X\in \mathcal{X}^{\perp}$, we have $\xi xt_{\xi}^{1}(X',X)=0$, hence the above $\mathbb{E}$-triangle is split.
This shows that $X$ is a direct summand of $W$, and hence $X\in \omega$.
Thus $\mathcal{X}\cap\mathcal{X}^{\perp}\subseteq\omega$.
\end{proof}

The following result is the so-called Auslander-Buchweitz approximation.

\begin{proposition}\label{prop-injective-2}{\rm (c.f. \cite[Theorem 1.1]{AB89T}, \cite[Theorem 2.8]{BMP19F}, \cite[Theorem 3.7]{MDZ})}
Let $\mathcal{X}$ and $\omega$ be  subcategories of $\mathscr{C}$. Assume that
\begin{itemize}
  \item [(1)] $\mathcal{X}$ is closed under $\xi$-extensions,
  \item [(2)] $\omega$ is a $\xi$-cogenerator in $\mathcal{X}$.
\end{itemize}
Then for any $C\in \mathcal{X}^{\wedge}_n$, there exist the following $\mathbb{E}$-triangles
$$\xymatrix@C=0.5cm{Y_{C}\ar[r]&X_{C}\ar[r]&C\ar@{-->}[r]&},$$
$$\xymatrix@C=0.5cm{C\ar[r]&Y^{C}\ar[r]&X^{C}\ar@{-->}[r]&}$$
in $\xi$ with $Y_{C}\in \omega^{\wedge}_{n-1}$, $Y^{C}\in \omega^{\wedge}_{n}$ and $X_{C},\ X^{C}\in \mathcal{X}$.
In particular, if $\omega$ is $\mathcal{X}$-injective, then $X_{C}\ra C$ is a right $\mathcal{X}$-approximation of $C$.
\end{proposition}

\begin{proof}
The proof is similar to that of \cite[Theorem 3.7]{MDZ}.
\end{proof}

\begin{corollary}\label{cor-4.5}
Let $\mathcal{X}$ and $\omega$ be  subcategories of $\mathscr{C}$. Assume that
\begin{itemize}
  \item [(1)] $\mathcal{X}$
is closed under $\xi$-extensions,
  \item [(2)] $\omega$ is closed under direct summands in $\mathscr{C}$,
  \item [(3)] $\omega$ is $\mathcal{X}$-injective and a $\xi$-cogenerator for $\mathcal{X}$.
\end{itemize}
Then
$$\omega^{\wedge}=\mathcal{X}^{\perp}\cap{\mathcal{X}{^\wedge}}.$$
\end{corollary}

\begin{proof}
By Proposition \ref{prop-injective-1}(1), $\omega^{\wedge}\subseteq\mathcal{X}^{\perp}$.
Moreover, $\omega^{\wedge}\subseteq\mathcal{X}^{\wedge}$.
Thus $\omega^{\wedge}\subseteq\mathcal{X}^{\perp}\cap{\mathcal{X}{^\wedge}}$.

Now let $C\in\mathcal{X}^{\perp}\cap{\mathcal{X}^{\wedge}}$.
By Proposition \ref{prop-injective-2}, there is an $\mathbb{E}$-triangle
$\xymatrix@C=0.5cm{Y\ar[r]&X\ar[r]&C\ar@{-->}[r]&}$
in $\xi$ with $X\in \mathcal{X}$ and $Y\in \omega^{\wedge}\subseteq\mathcal{X}^{\perp}$.
Since $C\in \mathcal{X}^{\perp}$, one has $X\in \mathcal{X}^{\perp}$, and hecne $X\in \mathcal{X}\cap \mathcal{X}^{\perp}$.
By Proposition \ref{prop-injective-1}(2), $X\in\omega$.
So $C\in \omega^{\wedge}$, and thus
$\mathcal{X}^{\perp}\cap{\mathcal{X}{^\wedge}}\subseteq\omega^{\wedge}$. Therefore, $\omega^{\wedge}=\mathcal{X}^{\perp}\cap{\mathcal{X}{^\wedge}}.$
\end{proof}

\begin{lemma}\label{thm-2}
Let $\omega$ be a subcategory of $\mathscr{C}$. Assume that
$\omega$ is closed under direct summands.
Then
\begin{itemize}
\item[(1)] $^{\perp}\omega={^{\perp}}(\omega^{\wedge})$.
\item[(2)] $\omega^{\wedge}\cap{^{\perp}}\omega\subseteq\omega$.
\end{itemize}
\end{lemma}

\begin{proof}
(1) Clearly, ${^{\perp}}(\omega^{\wedge})\subseteq{^{\perp}\omega}$.

Conversely, for any $X\in{^{\bot}}\omega,~\xi xt^{i\geq1}(X,\omega)=0$. For any $Y\in \omega^{\wedge}$, assume $\omega\text{-}\res Y=m<\infty$. Then there is a $\xi$-exact complex
$$\xymatrix@=0.5cm{W_{m}\ar[r]&W_{m-1}\ar[r]&\cdots\ar[r]&W_{1}\ar[r]&W_{0}\ar[r]&Y}$$ with each $W_{i}\in \omega,~0\leq i\leq m$. Applying the functor $\Hom_{\mathscr{C}}(X,-)$ to the $\mathbb{E}$-triangle $$\xymatrix@=0.5cm{W_{m}\ar[r]&W_{m-1}\ar[r]&K_{m-1}\ar@{-->}[r]&},$$
we can get the following exact sequence
$$\xymatrix@=0.5cm{\cdots\ar[r]&\xi xt^{i}_{\xi}(X,W_{m})\ar[r]&\xi xt^{i}_{\xi}(X,W_{m-1})\ar[r]&\xi xt^{i}_{\xi}(X,K_{m-1})\ar[r]&\xi xt^{i+1}_{\xi}(X,W_{m})\ar[r]&\cdots}.$$
 Since $\xi xt^{i}_{\xi}(X,W_{m-1})=\xi xt^{i+1}_{\xi}(X,W_{m})=0$ for any $i\geq1$, then $\xi xt^{i}_{\xi}(X,K_{m-1})=0$ for any $i\geq1$.
Let's take just $Y$ to be $K_{0}$. Repeat this process, we have $\xi xt^{i\geq1}_{\xi}(X,K_{j})=0,~0\leq j\leq m-1$, where $\mathbb{E}$-triangles $\xymatrix@=0.5cm{K_{k+1}\ar[r]&W_{k}\ar[r]&K_{k}\ar@{-->}[r]&}$ are in $\xi$, $~1\leq k\leq m-2$. Thus $\xi xt^{i\geq1}_{\xi}(X,Y)=0$, and so
${^{\perp}}(\omega^{\wedge})\supseteq{^{\perp}\omega}$

(2) Let $Z\in \omega^{\wedge}\cap{^{\bot}}\omega$. Then there is a $\xi$-exact complex $$\xymatrix@=0.5cm{W_{n}\ar[r]&W_{n-1}\ar[r]&\cdots \ar[r]&W_{1}\ar[r]&W_{0}\ar[r]&Z}$$ with $W_{i}\in\omega,~0\leq i\leq n$. Applying the functor ${\rm Hom}_{\mathscr{C}}(Z,-)$ to the $\mathbb{E}$-triangle $$\xymatrix@=0.5cm{W_{n}\ar[r]&W_{n-1}\ar[r]&K_{n-1}\ar@{-->}[r]&}$$
we can get the following exact sequence
$$\xymatrix@=0.5cm{\cdots\ar[r]&\xi xt^{1}_{\xi}(Z,W_{n-1})
\ar[r]&\xi xt^{1}_{\xi}(Z,K_{n-1})\ar[r]&\xi xt^{2}(Z,W_{n})}.$$
Then $\xi xt^{1}_{\xi}(Z,K_{n-1})=0$ since $\xi xt^{1}_{\xi}(Z,W_{n-1})=\xi xt^{2}_{\xi}(Z,W_{n})=0$. For the same reason, we can obtain that $\xi xt^{1}_{\xi}(Z,K_{i})=0,~1\leq i\leq n-1$. Then $\xi xt^{1}_{\xi}(Z,K_{1})=0$, that is, the $\mathbb{E}$-triangle $\xymatrix@=0.5cm{K_{1}\ar[r]&W_{0}\ar[r]&Z\ar@{-->}[r]&}$ is split. Thus $Z\oplus K_{1}\cong W_{0}\in\omega$. It follows that $Z\in\omega$ from the fact that $\omega$ is closed under direct summands.
\end{proof}

\begin{proposition}\label{theorem-2}
Let $\mathcal{X}$ and $\mathcal{Y}$ be subcategories of $\mathscr{C}$.
Assume that
\begin{itemize}
\item[(a)] $\mathcal{X}$ and $\mathcal{Y}$ are closed under direct summands and
$\mathcal{Y}\subseteq\mathcal{X}^{\wedge}$,
\item[(b)] $\mathcal{X}$ is closed under $\xi$-extensions and cocones of $\xi$-deflations, and
\item[(c)] $\mathcal{Y}$ is closed under $\xi$-extensions and cones of $\xi$-inflations.
\end{itemize}

If $\omega:=\mathcal{X}\cap \mathcal{Y}$ is  $\mathcal{X}$-injective and a $\xi$-cogenerator for $\mathcal{X}$, then
\begin{itemize}
  \item [(1)] $\mathcal{Y}=\omega^{\wedge}=\mathcal{X}^{\wedge}\cap\mathcal{X}^{\perp}=\mathcal{X}^{\wedge}
\cap\mathcal{X}^{\perp_{1}},$
  \item [(2)] $\mathcal{X}={^{\bot}}\omega\cap\mathcal{X}^{\wedge}={^{\bot}}(\omega^{\wedge})\cap\mathcal{X}^{\wedge}$.
\end{itemize}
\end{proposition}

\begin{proof}
(1)  Since $\mathcal{Y}$ is closed under cones of $\xi$-inflations, we have $\mathcal{Y}^{\wedge}=\mathcal{Y}$.
It follows that $\omega^{\wedge}\subseteq\mathcal{Y}$ since $\omega\subseteq\mathcal{Y}$.
Now let $Y\in\mathcal{Y}$. Since
$\mathcal{Y}\subseteq\mathcal{X}^{\wedge}$ by assumption,  there is an $\mathbb{E}$-triangle
$\xymatrix@C=0.5cm{K\ar[r]&X\ar[r]&Y\ar@{-->}[r]&}$
in $\xi$ with $X\in \mathcal{X}$ and $K\in \omega^{\wedge}\subseteq \mathcal{Y}$ by Proposition \ref{prop-injective-2}.
Since $\mathcal{Y}$ is closed under $\xi$-extensions, we have $X\in \mathcal{Y}$.
So $X\in \mathcal{X}\cap \mathcal{Y}=\omega$, and hence $Y\in \omega^{\wedge}$ and $\mathcal{Y}\subseteq\omega^{\wedge}$.
Thus $\mathcal{Y}=\omega^{\wedge}$.

By Corollary \ref{cor-4.5}, we have $\omega^{\wedge}=\mathcal{X}^{\perp}\cap{\mathcal{X}{^\wedge}}$.

Clearly, $\mathcal{X}^{\wedge}\cap\mathcal{X}^{\perp}\subseteq\mathcal{X}^{\wedge}\cap\mathcal{X}^{\perp_{1}}$.
Now let $Z\in\mathcal{X}^{\wedge}\cap\mathcal{X}^{\perp_{1}}$.
By Proposition \ref{prop-injective-2}, there is an $\mathbb{E}$-triangle
$\xymatrix@C=0.5cm{Z\ar[r]&W\ar[r]&X\ar@{-->}[r]&}$
in $\xi$ with $X\in \mathcal{X}$ and $W\in \omega^{\wedge}$.
Since $Z\in\mathcal{X}^{\perp_{1}}$, the above $\mathbb{E}$-triangle is split.
So $Z$ is a direct summand of $W$.
Notice that $\omega^{\wedge}(=\mathcal{Y}$) is closed under direct summands,
we have $Z\in \omega^{\wedge}=\mathcal{X}^{\wedge}\cap\mathcal{X}^{\perp}$.
Thus  $\mathcal{X}^{\wedge}\cap\mathcal{X}^{\perp}=\mathcal{X}^{\wedge}
\cap\mathcal{X}^{\perp_{1}}$.

(2) Clearly, $\mathcal{X}\subseteq {^{\bot}}\omega \cap \mathcal{X}^{\wedge}$. 
 By Lemma \ref{thm-2}(1), ${^{\bot}}\omega\cap\mathcal{X}^{\wedge}={^{\bot}}(\omega^{\wedge})\cap\mathcal{X}^{\wedge}$.
 It suffices to show that ${^{\bot}}\omega\cap\mathcal{X}^{\wedge}\subseteq\mathcal{X}$.

Now, let $M\in {^{\perp}\omega}\cap\mathcal{X}^{\wedge}$.
By Proposition \ref{prop-injective-2}, there is a $\mathbb{E}$-triangle
$\xymatrix@C=0.5cm{K\ar[r]&X\ar[r]&M\ar@{-->}[r]&}$ in $\xi$ with $X\in\mathcal{X}\subseteq{^{\bot}}\omega$ and $K\in \omega^{\wedge}$.
Then $K\in{^\perp\omega}$, and so $K\in \omega^{\wedge}\cap {^{\perp}\omega}\subseteq\omega$
by Lemma \ref{thm-2}.
Notice that
$\xi xt_{\mathcal{\xi}}^{1}(M,K)=0$, that is, the above $\mathbb{E}$-triangle is split, hence $X\cong K\oplus M$.
It follows that $M\in \mathcal{X}$ from the fact that
$\mathcal{X}$ is closed under direct summands.
Thus ${\mathcal{X}}^{\wedge}\cap {^{\perp}\omega}\subseteq\mathcal{X}$.
\end{proof}

Now we generalize the notion of left Frobenius pairs in abelian categories \cite{BMP19F} and in triangulated categories \cite{MZ} to  that in  extriangulated categories.

\begin{definition}
A pair of subcategories $(\mathcal{X},\omega)$ in $\mathscr{C}$ is called a \emph{left Frobenius pair} if
\begin{itemize}
\item[(LF1)] $\mathcal{X}$ and $\omega$ are closed under direct summands.
 \item[(LF2)]$\mathcal{X}$ is closed under $\xi$-extensions and cocones of $\xi$-deflations.
\item[(LF3)] $\omega$ is  $\mathcal{X}$-injective and a $\xi$-cogenerator of $\mathcal{X}$.
\end{itemize}
\end{definition}

We have the following example.

\begin{example}\label{exa-4.7}
\begin{itemize}
\item[]
\item[(1)]  $(\mathcal{GP}(\xi),\mathcal{P}(\xi))$ is a left Frobenius pair in $\mathscr{C}$ (see \cite[Lemma 4.16 and Theorem 4.17]{HZZ}).
\item[(2)] Given a left Frobenius pair $(\mathcal{X},\omega)$ in $\mathscr{C}$. If
$\mathcal{X}\text{-}\res \mathscr{C}=n<\infty$, then
$(\mathcal{X},\omega)$ is a left $n$-cotorsion pair in $\mathscr{C}$ by Proposition \ref{prop-injective-2}.
In particular, if $\sup\{\xi\text{-}\mathcal{G}\pd T\mid T\in \mathscr{C}\}=n<\infty$,
then $(\mathcal{GP}(\xi),\mathcal{P}(\xi))$ is a left $n$-cotorsion pair in $\mathscr{C}$.
\end{itemize}
\end{example}

In the following, we will study the closure properties of $\mathcal{X}^{\wedge}$ for a left Frobenius pair $(\mathcal{X},\omega)$ in $\mathscr{C}$.

\begin{lemma}\label{lemma-3-2}
Let $(\mathcal{X},\omega)$ be a left Frobenius pair in $\mathscr{C}$, and let
\begin{align}\label{1-tri}
\xymatrix@C=0.5cm{X\ar[r]&Y\ar[r]&Z\ar@{-->}[r]^{\delta_{1}}&}
\end{align}
be an $\mathbb{E}$-triangle in $\xi$.
\begin{itemize}
\item[(1)] Assume that $Z\in \mathcal{X}$. Then $X\in \mathcal{X}^{\wedge}$ if and only if $Y\in \mathcal{X}^{\wedge}$.
\item[(2)] Assume that $Y\in \mathcal{X}$. Then $X\in \mathcal{X}^{\wedge}$ if and only if $Z\in \mathcal{X}^{\wedge}$.
\end{itemize}
\end{lemma}

\begin{proof}
(1) The ``If" part. Assume that $Y\in \mathcal{X}^{\wedge}$. We may set $\mathcal{X}\text{-}\res Y=m<\infty$.
The case  $m=0$ is clear. Now assume that $m\geq 1$. Then there is an $\mathbb{E}$-triangle
$\xymatrix@C=0.5cm{K\ar[r]&X_{0}\ar[r]&Y\ar@{-->}[r]^{\delta_{2}}&}$
in $\xi$ with $X_{0}\in \mathcal{X}$ and $K\in \mathcal{X}_{m-1}^{\wedge}$. By (ET4)$^{op}$, we have
 a commutative diagram as follows:
\begin{align}\label{diag-1}
\xymatrix@=0.5cm{K\ar@{=}[d]\ar@{.>}[r]^{k}&X'_{0}\ar@{.>}[d]
\ar@{.>}[r]^{f}&X\ar@{-->}[r]\ar[d]^{g}&\\
K\ar[r]&X_{0}\ar[r]\ar@{.>}[d]&Y\ar@{-->}[r]^{\delta_{2}}\ar[d]^{h}&\\
&Z\ar@{-->}[d]^{\delta_{3}}\ar@{=}[r]&Z\ar@{-->}[d]^{\delta_{1}}\\&&&}
\end{align}
where $h^{*}\delta_{3}=k_{*}\delta_{2},~f_{*}\delta_{3}=\delta_{1}.$

Applying Lemma \ref{lem1}
yields the following commutative diagram:$$\xymatrix@=0.5cm{&X\ar@{=}[r]\ar@{.>}[d]&X\ar[d]\\X'_{0}
\ar@{=}[d]\ar@{.>}[r]&M\ar@{.>}[r]
\ar@{.>}[d]&Y\ar@{-->}[r]^{h^*\delta_3}\ar[d]^h&\\
X'_{0}\ar[r]&X_{0}\ar@{-->}[d]\ar[r]&Z\ar@{-->}[d]^{\delta_{1}}\ar@{-->}[r]^{\delta_{3}}&\\&&&}$$
where the $\E$-triangle $\xymatrix@=0.5cm{X'_{0}\ar[r]&M\ar[r]&Y\ar@{-->}[r]^{h^{*}\delta_{3}}&}$ belongs to $\xi$ by the equality $h^{*}\delta_{3}=k_{*}\delta_{2}$. Then $\xymatrix@=0.5cm{X'_{0}\ar[r]&X_{0}\ar[r]&Z\ar@{-->}[r]^{\delta_{3}}&}$  belongs to $\xi$ since $\xi$ is saturated. It follows that $X'_{0}\in \mathcal{X}$ since $X_{0},Z\in \mathcal{X}$. Thus $X\in\mathcal{X}^{\wedge}$ by the top $\mathbb{E}$-triangle in the diagram (\ref{diag-1}).

The ``only if" part. Assume that $X\in \mathcal{X}^{\wedge}$. We may set $\mathcal{X}\text{-}\res X=m<\infty$.
The case  $m=0$ is clear. Assume that $m\geq 1$,
By Proposition \ref{prop-injective-2}, there is an  $\mathbb{E}$-triangles
\begin{align}\label{1-tri1}
\xymatrix@=0.5cm{K\ar[r]&X_{0}\ar[r]^p&X\ar@{-->}[r]&}
\end{align}
in $\xi$ with $X_{0}\in \mathcal{X}$ and $K\in \omega^{\wedge}_{m-1}$. Since $Z\in{^{\bot}(\omega^{\wedge})}$ by Corollary \ref{cor-4.5}, applying the functor $\Hom_{\mathscr{C}}(Z,-)$ to the $\mathbb{E}$-triangle(\ref{1-tri1})
yields an isomorphism  $p_*:\xi xt_{\xi}^{1}(Z,X_{0})\to \xi xt_{\xi}^{1}(Z,X)$.
Then we have the following commutative diagram
$$\xymatrix{
X _{0} \ar[r]\ar[d]^p&X''\ar[r]\ar[d]&Z\ar@{-->}[r]^{\delta'}\ar@{=}[d]&\\
X \ar[r]&Y\ar[r]&Z\ar@{-->}[r]^-{\delta=p_*\delta'}&,}$$
where all $\mathbb{E}$-triangles are in $\xi$.
Since $\mathcal{X}$ is closed under $\xi$-extensions, $X''\in \mathcal{X}$.
By (ET4), we have the following commutative diagram
$$
\xymatrix@=0.5cm{
   K \ar[r]\ar@{=}[d]&X_0\ar[r]^p\ar[d]&X\ar@{-->}[r]\ar@{.>}[d]&\\
    K \ar@{.>}[r] & X'' \ar@{.>}[r]\ar[d] &B\ar@{.>}[d]\ar@{-->}[r]&\\
    &Z\ar@{=}[r]\ar@{-->}[d]^{\delta'} &Z\ar@{-->}[d]^{p_*\delta'}&\\
    &&& }
$$
By the second horizontal $\mathbb{E}$-triangle, we know $B\in \mathcal{X}^{\wedge}$. Then $Y\cong B\in\mathcal{X}^{\wedge}$.

(2) The ``only of" part is clear.

The ``if" part. Assume that $Z\in \mathcal{X}^{\wedge}$. We may set $\mathcal{X}\text{-}\res Z=m<\infty$.
We proceed by induction on $m$. The case $m=0$ is clear. Now assume that $m\geq 1$.
Then there is an $\mathbb{E}$-triangle
$\xymatrix@=0.5cm{K\ar[r]&X_{0}\ar[r]&Z\ar@{-->}[r]&}$
in $\xi$ with $X_{0}\in\mathcal{X}$ and $K\in \mathcal{X}^{\wedge}_{m-1}$.
By Lemma \ref{lem1}, we get the following commutative diagram
\begin{align}\label{diag-2}
\xymatrix@=0.5cm{&K\ar@{=}[r]\ar@{.>}[d]&K\ar[d]&\\
X  \ar@{.>}[r]\ar@{=}[d]&U\ar@{.>}[r]\ar@{.>}[d]&X_{0}\ar@{-->}[r]\ar[d]&\\
X \ar[r]&Y\ar[r]\ar@{-->}[d]&Z\ar@{-->}[r]\ar@{-->}[d]&\\
&&&}
\end{align}
Since $\xi$ is closed under base change, the middle row and the middle column are in $\xi$. Note that $Y,X_{0}\in \mathcal{X}$.
By (1) and the middle column in the above diagram, we have $U\in \mathcal{X}^{\wedge}$. Finally, by  (1) and the middle row in the above diagram, $X\in \mathcal{X}^{\wedge}$.
\end{proof}

\begin{proposition}\label{prop-X}
Let $(\mathcal{X},\omega)$ be a left Frobenius pair in $\mathscr{C}$.
For any $T\in \mathscr{C}$, the following statements are equivalent.
\begin{itemize}
\item[(1)] $\mathcal{X}$-$\res T\leq n$.
\item[(2)] If
\begin{equation}\label{4.7}
  \xymatrix@C=0.5cm{K_{n}\ar[r]&X_{n-1}\ar[r]&\cdots\ar[r]&X_{1}\ar[r]&X_{0}\ar[r]&T}
\end{equation}
is a $\xi$-exact complex in $\mathscr{C}$ with $X_{i}\in \mathcal{X}$ for any $0\leq i\leq n-1$,
then $K_{n}\in \mathcal{X}$.
\end{itemize}
\end{proposition}

\begin{proof}
$(2)\Longrightarrow(1)$ It is obvious.

$(1)\Longrightarrow(2)$ By Lemma \ref{lemma-3-2}, we have $K_{n}\in \mathcal{X}^{\wedge}$.
Since $\omega$ is $\mathcal{X}$-injective, that is, $\mathcal{X}\subseteq{^{\bot}}\omega$, then  we have
$\xi xt_{\xi}^{i}(K_{n},W)\cong \xi xt_{\xi}^{n+i}(T,W)=0$   
for any $W\in \omega$ and
 all $i\geq 1$. This shows that $K_{n}\in {^{\perp}\omega}$, and hence
$K_{n}\in\mathcal{X}^{\wedge}\cap{^{\perp}\omega}=\mathcal{X}$ by Proposition \ref{theorem-2}.
\end{proof}

The following result shows that $\mathcal{X}^{\wedge}$ satisfies the
two-out-of-three property.

\begin{theorem}\label{thm-4.9} {\rm (c.f. \cite[Propositions 3.10 and 3.14]{MDZ})}
Let $(\mathcal{X},\omega)$ be a left Frobenius pair in $\mathscr{C}$.
Then $\mathcal{X}^{\wedge}$ is closed under $\xi$-extensions, cocones of $\xi$-deflations , cones of $\xi$-inflations and direct summands.
\end{theorem}

\begin{proof}
Let
\begin{align}\label{2-tri}
\xymatrix@C=0.5cm{X\ar[r]&Y\ar[r]&Z\ar@{-->}[r]&}
\end{align}
be an $\mathbb{E}$-triangle in $\xi$.

(1) $\mathcal{X}^{\wedge}$ is closed under $\xi$-extensions. The proof is similar to that of \cite[Proposition 3.10]{MDZ}.

(2) $\mathcal{X}^{\wedge}$ is closed under cocones of $\xi$-deflations.
Indeed, let $Y,Z\in \mathcal{X}^{\wedge}$ with $\mathcal{X}\text{-}\res Z=n<\infty$. We proceed by induction on $n$.
The case for $n=0$ follows from Lemma \ref{lemma-3-2}(1). Now assume that $n\geq 1$. Then we have an $\mathbb{E}$-triangle
$\xymatrix@C=0.5cm{K\ar[r]&X_{0}\ar[r]&Z\ar@{-->}[r]&}$
in $\xi$ with $X_{0}\in\mathcal{X}$ and $K\in \mathcal{X}^{\wedge}_{n-1}$.
By Lemma \ref{lem1}, we have the following commutative diagram
$$\xymatrix@=0.5cm{&K\ar@{=}[r]\ar@{.>}[d]&K\ar[d]&\\
X  \ar@{.>}[r]\ar@{=}[d]&U\ar@{.>}[r]\ar@{.>}[d]&X_{0}\ar@{-->}[r]\ar[d]&\\
X \ar[r]&Y\ar[r]\ar@{-->}[d]&Z\ar@{-->}[r]\ar@{-->}[d]&\\
&&&}$$
where the second vertical and the second horizontal $\mathbb{E}$-triangles are in $\xi$.
By (1), we have $U\in \mathcal{X}^{\wedge}$. So $X\in \mathcal{X}^{\wedge}$ by Lemma \ref{lemma-3-2}(1).

Similarly, we can prove that $\mathcal{X}^{\wedge}$ is closed under cones of $\xi$-inflations.

(3) $\mathcal{X}^{\wedge}$ is closed under direct summands.
The proof is similar to that of \cite[Proposition 3.14]{MDZ}.
\end{proof}

Assume that $\mathcal{I}(\xi)$ is a cogenerating subcategory of $\mathscr{C}$ and $\mathcal{P}(\xi)$ is a generating subcategory of $\mathscr{C}$.
The following result provides
a method to construct (left) cotorsion pairs from left Frobenius pairs in $\mathscr{C}$ under the condition $\mathcal{X}^{\wedge}=\mathscr{C}$.

\begin{theorem}\label{thm-ftoc}

Assume that $\mathcal{I}(\xi)$ is a cogenerating subcategory of $\mathscr{C}$ and $\mathcal{P}(\xi)$ is a generating subcategory of $\mathscr{C}$.
Let $(\mathcal{X},\omega)$ be a left Frobenius pair in $\mathscr{C}$. Then the following statements are equivalent.
\begin{itemize}
\item[(1)] $\mathcal{X}^{\wedge}=\mathscr{C}$.
\item[(2)] $(\mathcal{X},\omega^{\wedge})$ is a cotorsion pair in $\mathscr{C}$ with $\xi \text{-}\id \omega<\infty$.
\item[(3)] $(\mathcal{X},\omega^{\wedge})$ is a left cotorsion pair in $\mathscr{C}$ with $\xi \text{-}\id \omega<\infty$.
\item[(4)] $\mathcal{X}={^{\perp}\omega}$ and $\xi \text{-}\id \omega<\infty$.
\end{itemize}
Moreover, if one of the above equivalent conditions holds, then $\mathcal{X}\text{-}\res \mathscr{C}=\xi$-$\id\omega$.
\end{theorem}

\begin{proof}
$(1)\Longrightarrow(2)$
Firstly, $\omega^{\wedge}=\mathcal{X}^{\perp}\cap\mathcal{X}^{\wedge}$ by Corollary \ref{cor-4.5}.
Secondly, $\mathcal{X}^{\wedge}$ is closed under direct summands by Proposition \ref{thm-4.9},
thus $\omega^{\wedge}$ is closed under direct summands. Moreover,
  $\xi xt_{\xi}^{1}(\mathcal{X},\omega^{\wedge})=0$ by Proposition \ref{prop-injective-1}.
Finally,
we can get that  $(\mathcal{X},\omega^{\wedge})$ is a cotorsion pair in $\mathscr{C}$ by Proposition \ref{prop-injective-2}.

Next let $W\in \omega$ and $T\in \mathscr{C}$ with $\mathcal{X}\text{-}\res T= n<\infty$.
Then there is a $\xi$-exact complex
$$\xymatrix@C=0.5cm{X_{n}\ar[r]&X_{n-1}\ar[r]&\cdots \ar[r]&X_{1}\ar[r]&X_{0}\ar[r]&T}$$
in $\mathscr{C}$ with $X_{i}\in \mathcal{X}\subseteq {^{\bot}}\omega$ for any $0\leq i\leq n$.
Applying the functor $\Hom_{\mathscr{C}}(-,W)$,  we can get
$$\xi xt_{\xi}^{n+1}(T,W)\cong \xi xt_{\xi}^{1}(X_{n},W)=0.$$
Hence $\xi\text{-}\id W\leq n$ by \cite[Lemma 3.9]{HZZZ}, and thus $\xi\text{-}\id \omega<\infty$.

$(2)\Longrightarrow(3)$ Obviously.

$(3)\Longrightarrow(4)$ By Remark \ref{remark-contra}, $\mathcal{X}={^{\perp_{1}}(\omega^{\wedge})}$.
Clearly, ${^{\perp}(\omega^{\wedge})}\subseteq {^{\perp_{1}}(\omega^{\wedge})}=\mathcal{X}$.
On the other hand,  $\mathcal{X}\subseteq {^{\perp}(\omega^{\wedge})}$
by Proposition \ref{prop-injective-1}.
Thus ${^{\perp_{1}}(\omega^{\wedge})}={^{\perp}(\omega^{\wedge})}$.
Clearly, ${^{\perp}(\omega^{\wedge})}= {^{\perp}\omega}$.
Thus $\mathcal{X}={^{\perp_{1}}(\omega^{\wedge})}= {^{\perp}\omega}$.

$(4)\Longrightarrow(1)$ Assume that $\xi\text{-}\id \omega=n<\infty$. For any $T\in \mathscr{C}$,
consider the following $\xi$-exact complex
$$\xymatrix@C=0.5cm{K\ar[r]&P_{n-1}\ar[r]&\cdots \ar[r]&P_{1}\ar[r]&P_{0}\ar[r]&T}$$
in $\mathscr{C}$ with $P_{i}\in \mathcal{P}(\xi)$ for any $0\leq i\leq n-1$.
For any $W\in \omega$,
 by \cite[Lemma 3.9]{HZZZ} and Remark \ref{remark-ds}, we have
$$\xi xt^{i}_{\xi}(K,W)\cong \xi xt^{n+i}_{\xi}(T,W)=0$$   
for any $i\geq 1$ since $\xi \text{-}\id W\leq n$.
So $K\in {^{\perp}\omega}$. Moreover, since $\mathcal{X}={^{\perp}\omega}$, we have that $K\in \mathcal{X}$.
It follows that $\mathcal{X}\text{-}\res T\leq n$ and then $T\in \mathcal{X}^{\wedge}$.
Therefore $\mathscr{C}=\mathcal{X}^{\wedge}$.
\end{proof}

Putting $\mathcal{X}=\mathcal{GP}(\xi)$ and $\omega=\mathcal{P}(\xi)$ in
Theorem \ref{thm-ftoc}, we can get an application in Gorenstein homological algebra as follows, in which
$(1)\Longrightarrow(2)$ was given in \cite[Proposition 4.7]{HZZ1}.

\begin{corollary}\label{cor-4.15}
Assume that $\mathcal{I}(\xi)$ is a cogenerating subcategory of $\mathscr{C}$ and $\mathcal{P}(\xi)$ is a generating subcategory of $\mathscr{C}$.
Then the following statements are equivalent.
\begin{itemize}
\item[(1)] $\sup\{\xi\text{-}\mathcal{G}\pd T\mid T\in \mathscr{C}\}< \infty$.
\item[(2)] $(\mathcal{GP}(\xi),{\mathcal{P}(\xi)}^\wedge)$ is a cotorsion pair
in $\mathscr{C}$ and $\xi\mbox{-}\id \mathcal{P}(\xi)<\infty$.
\item[(3)] $(\mathcal{GP}(\xi),{\mathcal{P}(\xi)}^\wedge)$
is a left cotorsion pair in $\mathscr{C}$ and $\xi\mbox{-}\id \mathcal{P}(\xi)<\infty$.
\item[(4)] $\mathcal{GP}(\xi)={^{\perp}\mathcal{P}(\xi)}$ and
$\xi\mbox{-}\id \mathcal{P}(\xi)<\infty$.
\end{itemize}
Moreover, if one of the above conditions holds, then
$\sup\{\xi\text{-}\mathcal{G}\pd T\mid T\in \mathscr{C}\}= \xi\mbox{-}\id \mathcal{P}(\xi)$.
\end{corollary}

Furthermore, we have the following result.

\begin{proposition}\label{prop-5}
Assume that $\mathcal{P}(\xi)$ is a generating subcategory of $\mathscr{C}$
and $\mathcal{I}(\xi)$ is a cogenerating subcategory of $\mathscr{C}$.
\begin{itemize}
\item[(1)] If $(\mathcal{GP}(\xi),\mathcal{P}(\xi))$ is a left $n$-cotorsion pair in $\mathscr{C}$, then
$$\sup\{\xi\text{-}\mathcal{G}\pd T\mid T\in \mathscr{C}\}=\xi\mbox{-}\id \mathcal{P}(\xi)\leq n.$$
Dually, if $(\mathcal{I}(\xi),\mathcal{GI}(\xi))$ is a right $m$-cotorsion pair in $\mathscr{C}$, then
$$\sup\{\xi\text{-}\mathcal{G}\id T\mid T\in \mathscr{C}\}=\xi\mbox{-}\pd \mathcal{I}(\xi)\leq m.$$
\item[(2)] If there are positive integers $n,\ m$ such that $(\mathcal{GP}(\xi),\mathcal{P}(\xi))$
is a left $n$-cotorsion pair, $(\mathcal{I}(\xi),\mathcal{GI}(\xi))$ is a right $m$-cotorsion pair
in $\mathscr{C}$ and $\mathscr{C}$ satisfies Condition(*), then we can choose $n=m=\xi\mbox{-}\id\mathcal{P}(\xi)= \xi\mbox{-}\pd \mathcal{I}(\xi)$.
\end{itemize}
\end{proposition}

\begin{proof}
(1) Assume that $(\mathcal{GP}(\xi),\mathcal{P}(\xi))$ is a left $n$-cotorsion pair in $\mathscr{C}$.
Then
$\sup\{\xi\text{-}\mathcal{G}\pd T\mid T\in \mathscr{C}\}\leq n$.
By Corollary \ref{cor-4.15}, we have
$$\sup\{\xi\text{-}\mathcal{G}\pd T\mid T\in \mathscr{C}\}=\xi\text{-}\id \mathcal{P}(\xi)\leq n.$$
Dually, we get the second assertion.

(2) By (1), we have $\sup\{\xi\text{-}\mathcal{G}\pd T\mid T\in \mathscr{C} \}=\xi\text{-}\id \mathcal{P}(\xi)\leq n$
and $\sup\{\xi\text{-}\mathcal{G}\id T\mid T\in \mathscr{C} \}=\xi\text{-}\pd \mathcal{I}(\xi)\leq m$, where $m,~n$ are positive integers. It follows from Remark \ref{2.29} that $\sup\{\xi\text{-}\mathcal{G}\pd T\mid T\in \mathscr{C} \}=\sup\{\xi\text{-}\mathcal{G}\id T\mid T\in \mathscr{C} \}
=\xi\mbox{-}\pd \mathcal{I}(\xi)=\xi\mbox{-}\id \mathcal{P}(\xi).$
\end{proof}

\begin{corollary}
Assume that $\mathcal{P}(\xi)$ is a generating subcategory of $\mathscr{C}$ and $\mathcal{I}(\xi)$
is a cogenerating subcategory of $\mathscr{C}$, and assume that $(\mathcal{GP}(\xi),\mathcal{P}(\xi))$
is a left $n$-cotorsion pair in $\mathscr{C}$, and $(\mathcal{I}(\xi),\mathcal{GI}(\xi))$ is a right $m$-cotorsion pair in $\mathscr{C}$. Then we have
\begin{itemize}
\item[(1)] $\mathscr{C}={\mathcal{GP}(\xi)}_{n}^\wedge={^{\perp_{n}}}({\mathcal{P(\xi)}}_{n-1}^\wedge)$.
\item[(2)] $\mathscr{C}={\mathcal{GI}(\xi)}_{m}^\vee=({\mathcal{I(\xi)}}_{m-1}^\vee){^{\perp_{m}}}$.
\end{itemize}
\end{corollary}

\begin{proof}
We only  prove (1), and (2) is  dual.
By Proposition \ref{prop-5}(1) $\mathscr{C}={\mathcal{GP}(\xi)}_{n}^\wedge$. By Theorem \ref{thm-1},
 $\mathcal{GP}(\xi)={^ {\perp_{1}}({\mathcal{P}(\xi)}_{n-1}^\wedge)}$.
Moreover, ${{\mathcal{P}(\xi)}^{\wedge}_{n-1}}\subseteq({{\mathcal{P}(\xi)}^{\wedge}_{n-1}})^{\wedge}_{n}$,
thus $(\mathcal{GP}(\xi),({{\mathcal{P}(\xi)}^{\wedge}_{n-1}})^{\wedge}_{n})$ is a left $(n+1)$-cotorsion pair
in $\mathscr{C}$. Then $\mathscr{C}={^{\perp_{n}}({\mathcal{P}(\xi)}^{\wedge}_{n-1})}$ by Corollary \ref{cor-2}.
\end{proof}

 Hashimoto \cite{H00A} introduced and studied relative Auslander-Buchweitz contexts in abelian categories. Ma et al. \cite{MZ} introduced and studied relative Auslander-Buchweitz contexts in triangulated categories with a proper class of triangles.
Motivated by it, we will introduce and study the  left (weak) Auslander-Buchweitz contexts with respect to $\xi$ in an extriangulated category.

\begin{definition}\label{def-leftAB}
Let $\mathcal{A}$ and $\mathcal{B}$ be subcategories of $\mathscr{C}$. Set $\omega:=\mathcal{A}\cap \mathcal{B}$.
We say that $(\mathcal{A},\mathcal{B})$ is a \emph{left weak Auslander-Buchweitz context}
(\emph{left weak $AB$ context} for short) in $\mathscr{C}$ if the following conditions are satisfied.
\begin{itemize}
\item[(LAB1)]  $(\mathcal{A},\omega)$ is a left Frobenius pair in $\mathscr{C}$.
\item[(LAB2)] $\mathcal{B}$ is closed under  direct summands, $\xi$-extensions and cones of $\xi$-inflations.
\item[(LAB3)] $\mathcal{B}\subseteq\mathcal{A}^{\wedge}$.
\end{itemize}
A left weak $AB$ context $(\mathcal{A},\mathcal{B})$ is called a \emph{left $AB$ context} if
the following condition is satisfied.
\begin{itemize}
\item[(LAB4)]
$\mathcal{A}^{\wedge}=\mathscr{C}.$
\end{itemize}
\end{definition}

The following proposition provides a method to obtain  left (weak) $AB$ contexts from left Frobenius pairs
in $\mathscr{C}$.

\begin{proposition}\label{prop-ftoc}
Let $(\mathcal{X},\omega)$ be a left Frobenius pair in $\mathscr{C}$.
Then $(\mathcal{X},\omega^{\wedge})$ is a left weak $AB$ context in $\mathscr{C}$.
Moreover, if $\mathcal{X}^{\wedge}=\mathscr{C}$, then $(\mathcal{X},\omega^{\wedge})$
is a left $AB$ context in $\mathscr{C}$.
\end{proposition}

\begin{proof}
By Proposition \ref{prop-injective-1},  $\mathcal{X}\cap\omega^{\wedge}=\omega$.
By Corollary \ref{cor-4.5},
 $\omega^{\wedge}=\mathcal{X}^{\perp}\cap \mathcal{X}^{\wedge}$. Moreover,
by Theorem \ref{thm-4.9},
$\omega^{\wedge}$ is closed under direct summands,
$\xi$-extensions and cones of $\xi$-inflations.
Clearly, $\omega^\wedge\subseteq\mathcal{X}^\wedge$. It follows that $(\mathcal{X},\omega^\wedge)$
is a left weak AB context in $\mathscr{C}$.
\end{proof}

The following  proposition provides a method  to obtain cotorsion pairs from left $AB$ contexts in $\mathscr{C}$.

\begin{proposition}\label{prop-2}
Let $(\mathcal{A},\mathcal{B})$ be a left weak $AB$ context in $\mathscr{C}$. Set
$\omega:=\mathcal{A}\cap \mathcal{B}$. Then
$$\omega=\mathcal{A}\cap\mathcal{A}^{\perp}\ \text{and}\ \omega^{\wedge}=\mathcal{B}.$$
Moreover, the following statements are equivalent.
\begin{itemize}
\item[(1)] $\mathcal{A}^{\wedge}=\mathscr{C}$.
\item[(2)] $(\mathcal{A},\mathcal{B})$ is a cotorsion pair in $\mathscr{C}$.
\end{itemize}
If one of the above conditions holds, then $\mathcal{A}$ is resolving.
\end{proposition}

\begin{proof}
By assumption,  $(\mathcal{A},\omega)$ is a left Frobenius pair in $\mathscr{C}$.
Then
$\omega=\mathcal{A}\cap\omega^{\wedge}$ and
$\omega^{\wedge}=\mathcal{A}^{\perp}\cap{\mathcal{A}{^\wedge}}$ by Proposition \ref{prop-injective-1} and Corollary \ref{cor-4.5}. Thus
$\omega=\mathcal{A}\cap\mathcal{A}^{\perp}\cap{\mathcal{A}{^\wedge}}
=\mathcal{A}\cap\mathcal{A}^{\perp}.$

For the second equality. Firstly, $\omega^{\wedge}\subseteq \mathcal{B}$.
Conversely, let $X\in \mathcal{B}\subseteq\mathcal{A}^{\wedge}$.
By Proposition \ref{prop-injective-2}, there is a $\mathbb{E}$-triangle
$\xymatrix@C=0.5cm{K\ar[r]&A\ar[r]&X\ar@{-->}[r]&}$
in $\xi$ with  $A\in \mathcal{A}$ and $K\in \omega^{\wedge}\subseteq\mathcal{B}$.
Then $A\in \mathcal{B}$.
Thus $A\in \mathcal{A}\cap\mathcal{B}=\omega$, and hence
$X\in \omega^{\wedge}$, which implies $\mathcal{B}\subseteq\omega^{\wedge}$.
Therefore, $\mathcal{B}=\omega^{\wedge}$.

Next we prove (1) $\Longleftrightarrow$ (2).

(1) $ \Longrightarrow$ (2) By Proposition \ref{prop-injective-1}, we have
$\mathcal{A}\subseteq {^{\perp}(\omega^{\wedge})}$ and $\xi xt_{\xi}^{1}(\mathcal{A},\mathcal{B})=0$.
Assume that $\mathscr{C}=\mathcal{A}^{\wedge}$. Then for any $T\in \mathscr{C}$, there are $\mathbb{E}$-triangles
$\xymatrix@C=0.5cm{B\ar[r]&A\ar[r]&T\ar@{-->}[r]&}$
and
$\xymatrix@C=0.5cm{T\ar[r]&B'\ar[r]&A'\ar@{-->}[r]&}$
in $\xi$ with $A,A'\in \mathcal{A}$ and $B,B'\in \omega^{\wedge}=\mathcal{B}$
by Proposition \ref{prop-injective-2}. Thus
$(\mathcal{A},\mathcal{B}=\omega^{\wedge})$ is a cotorsion pair in $\mathscr{C}$.

(2) $\Longrightarrow$ (1) Clearly.
\end{proof}

The following proposition provides a method to
obtain left Frobenius pairs and left (weak) $AB$ contexts from cotorsion pairs in $\mathscr{C}$.

\begin{proposition}\label{prop-ctowab}
Let
$(\mathcal{U},\mathcal{V})$ be a cotorsion pair in $\mathscr{C}$ with $\mathcal{U}$ a resolving subcategory. Set $\omega:=\mathcal{U}\cap \mathcal{V}$.
Then $(\mathcal{U},\omega)$ is a left Frobenius pair in $\mathscr{C}$.

Moreover, we have the following assertions.
\begin{itemize}
\item[(1)] If $\mathcal{V}\subseteq\mathcal{U}^{\wedge}$, then
$(\mathcal{U},\mathcal{V})$ is a left weak $AB$ context in $\mathscr{C}$.
\item[(2)] If $\mathcal{U}^{\wedge}=\mathscr{C}$, then  $(\mathcal{U},\mathcal{V})$
is a left $AB$ context in $\mathscr{C}$.
\end{itemize}
\end{proposition}

\begin{proof}
First of all,
$\mathcal{U}$ is closed under $\xi$-extensions and cocones of $\xi$-deflations, and $\omega:=\mathcal{U}\cap\mathcal{V}$ is closed under direct summands.
Moreover, by Corollary \ref{lemma-resolving}, $\mathcal{V}\subseteq\mathcal{U}^\perp$ and
$\omega\subseteq\mathcal{U}\cap \mathcal{U}^\perp$, that is, $\omega$ is $\mathcal{U}$-injective.
Now, let $U\in \mathcal{U}$. Then there is an $\mathbb{E}$-triangle
$\xymatrix@C=0.5cm{U\ar[r]&V'\ar[r]&U'\ar@{-->}[r]&}$
in $\xi$ with $U'\in \mathcal{U}$ and $V'\in \mathcal{V}$.
 Since
$\mathcal{U}$ is closed under $\xi$-extensions, then $V'\in \mathcal{U}\cap\mathcal{V}=\omega$, which shows that $\omega$ is a $\xi$-cogenerator in $\mathcal{U}$.
Therefore,  $(\mathcal{U},\omega)$ is a left Frobenius pair in $\mathscr{C}$.

(1) By Corollary \ref{lemma-resolving}, $\mathcal{V}$ is closed under $\xi$-extensions and cones
of $\xi$-inflations. Since $\mathcal{V}\subseteq\mathcal{U}^{\wedge}$ by assumption,
$(\mathcal{U},\mathcal{V})$ is a left weak $AB$ context in $\mathscr{C}$.

(2) It is clear by (1).
\end{proof}

Now we give the main result of this paper.

\begin{theorem}\label{main}
For an integer $n\geq 1$, consider the following classes:
\begin{align*}
\mathfrak{A}:=&\{\text{A pair } (\mathcal{X},\omega) \text{ in } \mathscr{C}\mid
(\mathcal{X},\omega)\text{ is a left } \text{Frobenius pair in } \mathscr{C}\},\\
\mathfrak{B}:=&\{\text{A pair } (\mathcal{A},\mathcal{B}) \text{ in } \mathscr{C}\mid
(\mathcal{A},\mathcal{B})\text{ is a left } \text{weak AB context}\},\\
\mathfrak{C}:=&\{\text{A pair } (\mathcal{U},\mathcal{V}) \text{ in } \mathscr{C}\mid
(\mathcal{U},\mathcal{V}) \text{ is a } \text{cotorsion pair in }
\mathscr{C}\text{ with }\mathcal{U}\text{ resolving and }\mathcal{V}\subseteq\mathcal{U}^{\wedge}\},\\
\mathfrak{D}:=&\{\text{A pair } (\mathcal{U},\mathcal{V}) \text{ in } \mathscr{C}\mid
(\mathcal{U},\mathcal{V}) \text{ is an }n \text{-cotorsion pair in }
\mathscr{C}\text{ with }\mathcal{U}\text{ resolving and }\mathcal{V}\subseteq\mathcal{U}^{\wedge}\}.
\end{align*}
Then we have
\begin{itemize}
\item[(1)] There is a one-to-one correspondence between $\mathfrak{A}$ and $\mathfrak{B}$
given by
\begin{align*}
\Phi:&\mathfrak{A}\longrightarrow\mathfrak{B},\ \ (\mathcal{X},\omega)
\mapsto (\mathcal{X},\ {\omega}^{\wedge}),\\
\Psi:&\mathfrak{B}\longrightarrow\mathfrak{A},\ \ (\mathcal{A},\mathcal{B})
\mapsto (\mathcal{A},\ \mathcal{A}\cap\mathcal{B}).
\end{align*}
\item[(2)] $\mathfrak{C}\subseteq\mathfrak{B}$.
\item[(3)] $\mathfrak{C}=\mathfrak{D}$.
\end{itemize}
\end{theorem}

\begin{proof}
(1) First of all, $\Phi$ is well-defined by Proposition \ref{prop-ftoc}.
Next we prove
$\Phi\Psi={\rm Id}_{\mathfrak{B}}$ and $\Psi\Phi={\rm Id}_{\mathfrak{A}}$.
Let $(\mathcal{A},\mathcal{B})\in \mathfrak{B}$. Then
$\Phi\Psi(\mathcal{A},\mathcal{B})=\Phi(\mathcal{A},\mathcal{A}\cap\mathcal{B})
=(\mathcal{A},(\mathcal{A}\cap\mathcal{B})^{\wedge}).$
By Proposition \ref{prop-2},  $\mathcal{B}=(\mathcal{A}\cap\mathcal{B})^{\wedge}$.
It follows that $\Phi\Psi(\mathcal{A},\mathcal{B})=(\mathcal{A},\mathcal{B})$
and $\Phi\Psi={\rm Id}_{\mathfrak{B}}$. Conversely,
let $(\mathcal{X},\omega)\in \mathfrak{A}$. Then
$\Psi\Phi(\mathcal{X},\omega)=\Psi(\mathcal{X},\omega^{\wedge})
=(\mathcal{X},\mathcal{X}\cap\omega^{\wedge}).$
Since $\mathcal{X}\cap\omega^{\wedge}=\omega$ by Proposition \ref{prop-injective-1},
we have $\Psi\Phi(\mathcal{X},\omega)=(\mathcal{X},\omega)$ and $\Psi\Phi={\rm Id}_{\mathfrak{A}}$.

(2) It follows from Proposition \ref{prop-ctowab}.

(3) It follows from Proposition \ref{prop-6}.
\end{proof}

Furthermore, we have the following theorem.

\begin{theorem}\label{cor}
For an integer $n\geq 1$, consider the following classes:
\begin{align*}
\mathfrak{A}':=&\{\text{A pair } (\mathcal{X},\omega) \text{ in } \mathscr{C}:
(\mathcal{X},\omega)\text{ is a left } \text{Frobenius pair with }\mathcal{X}^{\wedge}=\mathscr{C}\},\\
\mathfrak{B}':=&\{\text{A pair } (\mathcal{A},\mathcal{B}) \text{ in } \mathscr{C}:
(\mathcal{A},\mathcal{B})\text{ is a left } \text{AB context}\},\\
\mathfrak{C}':=&\{\text{A pair } (\mathcal{U},\mathcal{V}) \text{ in } \mathscr{C}:
(\mathcal{U},\mathcal{V}) \text{ is a } \text{cotorsion pair in } \mathscr{C}\text{ with }
\mathcal{U}\text{ resolving and }\mathcal{U}^{\wedge}=\mathscr{C}\},\\
\mathfrak{D}':=&\{\text{A pair } (\mathcal{U},\mathcal{V}) \text{ in } \mathscr{C}:
(\mathcal{U},\mathcal{V}) \text{ is an }n \text{-cotorsion pair in }
\mathscr{C}\text{ with }\mathcal{U}\text{ resolving and }\mathcal{U}^{\wedge}=\mathscr{C}\}.
\end{align*}
Then $\mathfrak{B}'=\mathfrak{C}'=\mathfrak{D}'$ and there is a one-to-one correspondence between
$\mathfrak{A}'$ and $\mathfrak{B}'$.
\end{theorem}

\begin{proof}
By Theorem \ref{main}, it suffices to show $\mathfrak{B}\subseteq \mathfrak{C}$.
Now the assertion follows from Proposition \ref{prop-2}.
\end{proof}

Following Example \ref{exa-4.7}(1) and Theorem  \ref{main}, we have that
$(\mathcal{GP}(\xi),\mathcal{P}(\xi)^{\wedge})$ is a left weak $AB$ context in $\mathscr{C}$.
In addition, if $\sup\{\xi\text{-}\mathcal{G}\pd T\mid T\in \mathscr{C}\}<\infty$,
then $(\mathcal{GP}(\xi),\mathcal{P}(\xi)^{\wedge})$ is a left $AB$ context by Theorem \ref{cor}.

Now let $\mathscr{C}=\mathcal{T}$ be a triangulated category, $[1]$  the shift functor, and $\mathbb{E}={\rm Hom}_{\mathcal{T}}(-,-[1])$. Then we recover Ma-Zhao-Huang's result as follows.

\begin{corollary}{\rm (\cite[Theorem 4.23]{MZ})}
Let $\mathscr{C}$ be a triangulated category with a proper class of triangles. For an integer $n\geq 1$, consider the following classes:
\begin{align*}
\mathfrak{A}':=&\{\text{A pair } (\mathcal{X},\omega) \text{ in }\mathcal{T}:
(\mathcal{X},\omega)\text{ is a left } \text{Frobenius pair with }\mathcal{X}^{\wedge}=\mathcal{T}\},\\
\mathfrak{B}':=&\{\text{A pair } (\mathcal{A},\mathcal{B}) \text{ in } \mathcal{T}:
(\mathcal{A},\mathcal{B})\text{ is a left } \text{AB context}\},\\
\mathfrak{C}':=&\{\text{A pair } (\mathcal{U},\mathcal{V}) \text{ in } \mathcal{T}:
(\mathcal{U},\mathcal{V}) \text{ is a } \text{cotorsion pair in } \mathcal{T}\text{ with }
\mathcal{U}\text{ resolving and }\mathcal{U}^{\wedge}=\mathcal{T}\},\\
\mathfrak{D}':=&\{\text{A pair } (\mathcal{U},\mathcal{V}) \text{ in } \mathcal{T}:
(\mathcal{U},\mathcal{V}) \text{ is an }n \text{-cotorsion pair in }
\mathscr{C}\text{ with }\mathcal{U}\text{ resolving and }\mathcal{U}^{\wedge}=\mathcal{T}\}.
\end{align*}
Then $\mathfrak{B}'=\mathfrak{C}'=\mathfrak{D}'$ and there is a one-to-one correspondence between
$\mathfrak{A}'$ and $\mathfrak{B}'$.
\end{corollary}

Finally, let $\mathscr{C}=\mathcal{A}$ be an abelian category, and $\mathbb{E}={\rm Ext}^1_{\mathcal{A}}(-,-)$. Then we recover partially
Becerril-Mendoza-P\'{e}rez-Santiago's result as follows.

\begin{corollary} {\rm (cf. \cite[Theorem 5.4]{BMP19F})}
Let $\mathcal{A}$ be an abelian category.  For an integer $n\geq 1$, consider the following classes:
\begin{align*}
\mathfrak{A}':=&\{\text{A pair } (\mathcal{X},\omega) \text{ in } \mathcal{A}:
(\mathcal{X},\omega)\text{ is a left } \text{Frobenius pair with }\mathcal{X}^{\wedge}=\mathcal{A}\},\\
\mathfrak{B}':=&\{\text{A pair } (\mathcal{A},\mathcal{B}) \text{ in } \mathcal{A}:
(\mathcal{A},\mathcal{B})\text{ is a left } \text{AB context}\},\\
\mathfrak{C}':=&\{\text{A pair } (\mathcal{U},\mathcal{V}) \text{ in } \mathcal{A}:
(\mathcal{U},\mathcal{V}) \text{ is a } \text{cotorsion pair in } \mathcal{A}\text{ with }
\mathcal{U}\text{ resolving and }\mathcal{U}^{\wedge}=\mathcal{A}\},\\
\mathfrak{D}':=&\{\text{A pair } (\mathcal{U},\mathcal{V}) \text{ in } \mathcal{A}:
(\mathcal{U},\mathcal{V}) \text{ is an }n \text{-cotorsion pair in }
\mathscr{C}\text{ with }\mathcal{U}\text{ resolving and }\mathcal{U}^{\wedge}=\mathcal{A}\}.
\end{align*}
Then $\mathfrak{B}'=\mathfrak{C}'=\mathfrak{D}'$ and there is a one-to-one correspondence between
$\mathfrak{A}'$ and $\mathfrak{B}'$.
\end{corollary}

\section*{Acknowledgment}
This work was funded by  the project ZR2021QA001 supported by Shandong Provincial Natural Science Foundation.

\end{document}